\documentclass[DIV=14,letterpaper]{scrartcl}
\pdfoutput=1

\date{}

\usepackage{tikz}
\usetikzlibrary{calc}
\usetikzlibrary{matrix}
\usepackage{amsmath,amssymb,amsfonts,amsthm,subfigure}
\usepackage{color}                    
\usepackage[pdftex]{hyperref}

\usepackage{graphicx}
\usepackage{epstopdf}
\usepackage{enumerate}
\usepackage{booktabs,multirow,threeparttable,makecell,caption} 
\usepackage{arydshln} 
\usepackage{bm} 

\allowdisplaybreaks[4]

\newcommand{\ud}{\mathrm{d}}

\newcommand{\bey}{\begin{eqnarray}}
\newcommand{\eey}{\end{eqnarray}}

\newcommand{\beq}{\begin{equation}}
\newcommand{\eeq}{\end{equation}}
\theoremstyle{plain}
\newtheorem{theorem}{\hspace{6mm}Theorem}[section]

\newtheorem{lemma}{\hspace{6mm}Lemma\,}[section]

\theoremstyle{definition}
\newtheorem{definition}{\hspace{6mm}Definition}[section]
\theoremstyle{remark}
\newtheorem{example}{\hspace{6mm}Example}[section]
\newtheorem{remark}{\hspace{6mm}Remark}[section]

\title{Novel superconvergence and ultraconvergence structures for the finite volume element method \thanks{The first and second authors are supported in part by the National Natural Science Foundation of China (No.12371396).}}

\author{Xiang Wang\thanks{School of Mathematics, Jilin University, Changchun 130012, China (wxjldx@jlu.edu.cn).}
\and Yuqing Zhang\thanks{School of Mathematics, Jilin University, Changchun 130012, China (zyq23@mails.jlu.edu.cn).}
\and Zhimin Zhang\thanks{Department of Mathematics, Wayne State University, Detroit, MI 48202, USA (ag7761@wayne.edu).}
}

\begin{document}
\maketitle

\begin{abstract}

This paper develops novel natural superconvergence and ultraconvergence structures for the bi-$k$-order finite volume element (FVE) method on
rectangular meshes. These structures furnish tunable and possibly asymmetric superconvergence and ultraconvergence points. We achieve
one-order-higher superconvergence for both derivatives and function values, and two-orders-higher ultraconvergence for derivatives--a
phenomenon that standard bi-$k$-order finite elements do not exhibit. Derivative ultraconvergence requires three conditions: a diagonal
diffusion tensor, zero convection coefficients, and the FVE scheme satisfying tensorial $k$-$k$-order orthogonality (imposed via dual mesh
constraints). This two-dimensional derivative ultraconvergence is not a trivial tensor-product extension of the one-dimensional phenomena; its
analysis is also considerably more complex due to directional coupling. Theoretically, we introduce the asymmetric-enabled M-decompositions
(AMD-Super and AMD-Ultra) to rigorously prove these phenomena. Numerical experiments confirm the theory.

\end{abstract}
\noindent{\textbf{ AMS 2020 Mathematics Subject Classification.} }
65N12, 65N08, 65N30

\noindent{\textbf{ Key Words.}}
ultraconvergence, superconvergence, finite volume element. 

\section{Introduction}
Superconvergence refers to the phenomenon where numerical solutions achieve higher convergence rates than those predicted by standard global error estimates.
Research in this area typically focuses on three categories: \textit{Natural superconvergence} refers to the numerical solution itself achieving higher convergence rates at specific points \cite{Babuska.2007,Cao.2013,Chen.1995,Lin.1996,Schatz.1996,Wang.2021b,Wang.2024,Zhu.1989}; \textit{Global superconvergence} describes the higher-order error estimate between the numerical solution and a superclose function in global norms \cite{Babuska.2007,Chou.2007,Krizek.1987,Wang.2021b}; and \textit{Post-processed superconvergence} denotes enhanced convergence rates achieved through recovery techniques \cite{Bramble.1977,Cockburn.2025,Lin.2013,Yang.2009,Zhang.2005,Zienkiewicz.1992}. As an effective way for improving numerical accuracy and efficiency, superconvergence has been extensively studied in the contest of the finite element method (FEM) \cite{Hu.2021,Lin.2008,Thomee.1977,Wahlbin.1995,Zhang.2012} and the finite volume element method (FVEM) \cite{Chen.2010,Chen.2015,Chou.2007,Li.2000,Li.2021,Sheng.2022}, among others.

The finite volume element method is valued for its local conservation properties and for the flexibility afforded by
 distinct trial and test spaces \cite{Cai.1991,Ewing.2002,Hong.2018,Huang.1998,Lv.2012b,Suli.1991,Wang.2016,Xu.2009,Zhang.2015}. This flexibility enables FVE schemes to achieve enhanced stability in convection-dominated problems \cite{Zhang.2023}, to provide optimal $L^2$ convergence rates using various dual mesh strategies \cite{Lin.2015,Lv.2012,Zhang.2023}, and to support tunable superconvergence structures \cite{Wang.2021b,Wang.2024}. Most existing
 natural \textit{superconvergence} results--narrowly defined as one-order-higher accuracy--for FVEM and FEM on rectangular meshes are restricted to pure diffusion problems, with superconvergence points being symmetric \cite{Cao.2015,Lv.2012b,Meng.2023}. \textit{Ultraconvergence} refers to superconvergence yielding at least two extra orders of accuracy \cite{Cao.2013,ChenC.2013,Douglas.1974,He.2016,Wang.2017,Wang.2024,Zhang.1996}.
Studies in FEM have identified natural function-value ultraconvergence at element vertices \cite{ChenC.2013,Douglas.1974} and post-processed ultraconvergence \cite{He.2016,Zhang.1996}. Notably, standard FEM exhibits no natural ultraconvergence for derivatives. For FVEM, natural derivative ultraconvergence has been established only in one dimension \cite{Cao.2013,Wang.2024}; on rectangular meshes it remains completely unexplored.

This paper develops novel natural superconvergence and ultraconvergence structures for bi-$k$-order FVE schemes on rectangular meshes applied to following general diffusion-convection-reaction problems.
\begin{align}\label{eq:BVP}
\left\{
\begin{array}{rl}
-\nabla\cdot(\mathbb{D}\nabla u) + \mathbb{Q}\cdot\nabla u + \boldsymbol{r} u = f, & \text{ in }\; \Omega:= (0,1)^2,\\
u = 0, & \text{ on }\; \partial \Omega,
\end{array}
\right.
\end{align}
where the diffusion tensor $\mathbb{D} = (d_{ij})_{2 \times 2}$ is symmetric and uniformly positive definite, the convection vector $\mathbb{Q} = (q_1, q_2)^\top$ and reaction coefficient $\boldsymbol{r}$ satisfy $\boldsymbol{r} - \frac{1}{2} \nabla \cdot \mathbb{Q} \geq \kappa > 0$ for some constant $\kappa$.

The novel results include $(k+1)$-order derivative superconvergence and $(k+2)$-order function-value superconvergence for full diffusion-convection-reaction equations, as well as $(k+2)$-order derivative ultraconvergence for diffusion-reaction equations with a diagonal diffusion tensor. Notably, this natural derivative ultraconvergence property is unattainable with traditional bi-$k$-order finite element methods. Furthermore, in contrast to the one-dimensional case, the FVE scheme on rectangular meshes does not admit derivative ultraconvergence beyond order $(k+2)$. The aforementioned phenomena require the FVE scheme to satisfy tensorial orthogonality conditions of sufficient order, which impose independent constraints on the dual strategies in the $x$- and $y$-directions. It is worth mentioning that a wide family of FVE schemes satisfies these tensorial orthogonality conditions, making the novel superconvergence and ultraconvergence points both tunable and potentially asymmetric.

Within each rectangular element $K$, the locations of these points are characterized as follows: (i) Derivative superconvergence points (say, for the $x$-partial derivative) lie on $k$ lines parallel to the $y$-axis, defined by the $k$ dual points in the $x$-direction; (ii) Function-value superconvergence points are the tensor product of the one-dimensional function-value superconvergence points in both directions, totaling $(k+1)^2$ points per element; (iii) Derivative ultraconvergence points (again for the $x$-derivative) are the tensor product of the $k$ dual points in the $x$-direction and the $(k+1)$ one-dimensional function-value superconvergence points in the $y$-direction, yielding $k(k+1)$ points per element.

Rigorous proofs rest on the proposed asymmetric-enabled M-decompositions (AMD-Super and AMD-Ultra). The analysis proceeds in two stages. First, we construct superclose functions $u_{I,Super}$ and $u_{I,Ultra}$ via AMD techniques. These superclose functions serve as pivotal bridges between the exact solution $u$ and the numerical solution $u_h$, yielding the desired natural super- and ultraconvergence at possibly asymmetric points.
Second, under the relevant tensorial orthogonality conditions we establish the global super- and ultraconvergence error bounds
$\|u_h - u_{I,Super}\|_{1} = \mathcal{O}(h^{k+1})$, $\|u_h - u_{I,Super}\|_{0} = \mathcal{O}(h^{k+2})$, and $\|u_h - u_{I,Ultra}\|_{1} = \mathcal{O}(h^{k+2})$.

The remainder of this paper is organized as follows. Section~\ref{sec:FVE Schemes} introduces the construction of bi-$k$-order FVE schemes and defines four interpolation operators for subsequent proofs; Section~\ref{sec:OC_AMD} establishes the tensorial $k$-$r$-order orthogonality condition and the asymmetric-enabled M-decompositions, which form the theoretical foundation; Section~\ref{sec:Preparation_estimation} derives the inf-sup condition for the FVE bilinear form and the difference estimate between FEM and FVE bilinear forms; derivative ultraconvergence is rigorously proven in Section~\ref{sec:Ultraconvergence}, while Section~\ref{sec:Superconvergence} provides detailed proofs of both derivative and function-value superconvergence; numerical validation of all theoretical results is presented in Section~\ref{sec:Numerical experiments}, with concluding remarks in Section~\ref{sec:Conclusion}.

Throughout this work, we employ the notation $A \lesssim B$ to indicate $A \leq B$ ignoring some positive constant independent of $A$ and $B$.

\section{FVE schemes}
\label{sec:FVE Schemes}
The FVE method involves both the primary mesh and the dual mesh, with corresponding trial and test function spaces defined accordingly. Here, we present the formulation of the FVE method on rectangular meshes considered in this paper.

\textbf{Primary mesh and trial function space.} Denote $\mathbb{Z}_{N}:=\{1,\dots,N\}$ and $\mathbb{Z}^{0}_{N}:=\{0\}\cup\mathbb{Z}_{N}$. Let $x_i$ ($i\in \mathbb{Z}_{N_x}^{0}$) and $y_j$ ($j\in \mathbb{Z}_{N_y}^{0}$) be the dividing points of $\Omega$ in each direction satisfying
\begin{align*}
0& = x_0 < x_1 < x_2 < \dots < x_{N_x} = 1,\\
0& = y_0 < y_1 < y_2 < \dots < y_{N_y} = 1.
\end{align*}
Then, we obtain the primary mesh (rectangular partition) $\mathcal{T}_{h}$ of $\Omega$ that
\begin{align*}
\mathcal{T}_{h} := \{K:\, K = [x_{i-1},\,x_{i}]\times[y_{j-1},y_{j}], \,\, i\in\mathbb{Z}_{N_x}, j\in\mathbb{Z}_{N_y} \},
\end{align*}
Suppose the primary mesh is quasi-uniform, i.e., there exists a positive constant $c_1$ such that
\[
h/h_{i}^x\leq c_1,\,\,h/h_{j}^y\leq c_1,  \quad \forall i\in\mathbb{Z}_{N_x},\,\, j\in\mathbb{Z}_{N_y},
\]
where, $h_{i}^x = x_{i}- x_{i-1}$, $h_{j}^y = y_{j}- y_{j-1}$, and $h = \max\limits_{i\in\mathbb{Z}_{N_x}, j\in\mathbb{Z}_{N_y}}\{h_{i}^x,\,h_{j}^y\}$ is the mesh size.

The trial function space is chosen as the bi-$k$-order ($k\geq1$) Lagrange finite element space
\begin{align*}
\mathit{U}_{h}^{k} := \{ w_{h}\in \mathit{C}(\Omega): \, w_{h}\vert_{K} \in Q^{k}(K),\, \forall K\in \mathcal{T}_{h} \quad\!\!\! \mathrm{and}\quad\!\!\! w_{h}|_{\partial \Omega}=0\}.
\end{align*}
Here, $Q^{k}(K)$ is the bi-$k$-order polynomial space on $K$.

\textbf{Dual mesh and test function space.} Let $\alpha^x_{s}$ ($s\in \mathbb{Z}_{k}$) and $\alpha^y_{t}$ ($t\in \mathbb{Z}_{k}$) be the dividing points of the reference element $\hat{K}:=[-1,1]^{2}$ in each direction satisfying
\begin{align*}   
\displaystyle
\left\{
\begin{array}{l}
 -1 < \alpha^x_{1} < \alpha^x_{2} < \cdots < \alpha^x_{k} < 1,\\
 -1 < \alpha^y_{1} < \alpha^y_{2} < \cdots < \alpha^y_{k} < 1.
\end{array}
\right.
\end{align*}
Denote $\boldsymbol{\alpha}^{x}:=(\alpha^x_{1},\alpha^x_{2},\cdots,\alpha^x_{k})$ and $\boldsymbol{\alpha}^{y}:=(\alpha^y_{1},\alpha^y_{2}, \cdots,\alpha^y_{k})$ the dual strategies in $x$-direction and $y$-direction respectively. The notation $\boldsymbol{\hat{\alpha}_{s,t}} = (\alpha^x_{s},\alpha^y_{t})$ $(s,t\in\mathbb{Z}_{k})$ is a dual point on $\hat{K}$. Then, given a primary element $K\in\mathcal{T}_{h}$, the dual points are defined by
\begin{align}\label{eq:Dual points}
\boldsymbol{\alpha_{s,t}^{K}} := (\alpha^{x}_{s,K},\alpha^{y}_{t,K}) = F_{K}(\boldsymbol{\hat{\alpha}_{s,t}}),
   \quad \forall s,t\in\mathbb{Z}_{k}.
\end{align}
Here, the from $\hat{K}$ to $K$ mapping $F_{K}:(\hat{x},\hat{y})\to(x,y)$ is given by
\begin{align}\label{eq:FK}
\left\{
  \begin{array}{l}
   \displaystyle x =\frac{x_{i}-x_{i-1}}{2}\hat{x}+\frac{x_{i}+x_{i-1}}{2},  \quad \forall \hat{x}\in[-1,1], \\
   \displaystyle y =\frac{y_{j}-y_{j-1}}{2}\hat{y}+\frac{y_{j}+y_{j-1}}{2},  \quad \forall \hat{y}\in[-1,1]. \\
  \end{array}
\right.
\end{align}
Thus, the dual mesh $\mathcal{T}_{h}^{*}$ of $\Omega$ is defined accordingly.

The test function space $V_{h}$ is the piecewise constant function space based on $\mathcal{T}_{h}^{*}$ that
\begin{align*}
V_{h} := \{v_h\in L^{2}(\Omega) : \,v_h=\sum_{K^{*}\in\mathcal{T}_{h}^{*}} v_{K^*}\psi_{K^*},\,\,
                                     \mathrm{and} \,\, v_{h}\!\!\mid_{\partial \Omega}=0\},
\end{align*}
where $v_{K^*}$ and $\psi_{K^*}=\chi(K^*)$ are the constant and the characteristic function on $K^{*}$, respectively.

\textbf{FVE schemes.}
The FVE scheme for solving (\ref{eq:BVP}) is to find $u_h \in U_{h}^{k}$, such that
\begin{align}\label{eq:FVEscheme}
a_h (u_h,v_h) = (f,v_h), \quad \forall v_h\in V_h,
\end{align}
where $(\cdot,\cdot)$ is the inner product, and $a_{h}(\cdot,\cdot)$ is the bilinear form of the FVE scheme that
\begin{align}\label{eq:ah_K}
a_{h}(u_{h},v_{h}) := \sum_{K \in\mathcal{T}_{h}}a^{K}_{h}(u_{h},v_{h}),
\end{align}
where
\begin{align*}
a^{K}_{h}(u_{h},v_{h}) = \sum_{K^{*}\in\mathcal{T}_{h}^{*}}
                           -\int_{\partial K^{*}\cap K}((\mathbb{D}\nabla u_{h})\cdot\boldsymbol{n})v_{h}\ud s
                     +   \int_{K^{*}\cap K}(\mathbb{Q}\cdot\nabla u_{h}+\boldsymbol{r} u_{h})v_{h}\ud x \ud y.
\end{align*}
Here, $\boldsymbol{n}$ denotes the unit outer normal vector to $\partial K^{*}$.

\subsection{Interpolation operators for analysis}
\label{subsubsec:Some notations}
For analysis purpose, we define the following four operators $\Pi_{h}^{k}$, $\Pi_{K}^{0}$, $\Pi_{K}^{1}$ and $\Pi_{h}^{k,*}$ $(K\in \mathcal{T}_{h})$.

$\bullet$ $\Pi_{h}^{k}: H_{0}^{1}(\Omega)\rightarrow U_{h}^{k}$, the piecewise bi-$k$-order Lagrange interpolation operator based on $\mathcal{T}_{h}$.

$\bullet$ $\Pi_{K}^{0}: L^{2}(K)\rightarrow P^{0}(K)$, the average operator on $K\in\mathcal{T}_{h}$. For any $w\in L^{2}(K)$, $\bar{w}_{K}:=\Pi_{K}^{0}w$ is the average of $w$ on $K$.

$\bullet$ $\Pi_{K}^{1}: L^{2}(K)\rightarrow P^{1}(K)$, a linear interpolation operator on $K\in\mathcal{T}_{h}$. Let $\displaystyle\xi  = \frac{x-x_{i-1}}{x_{i}-x_{i-1}}$ and $\displaystyle\eta = \frac{y-y_{j-1}}{y_{j}-y_{j-1}}$. For $w_{h} \in U_{h}^{k}$, we define
\begin{align}\label{eq:w1_def}
\Pi_{K}^{1}w_{h} =&  w_{h}(x_{i-1},y_{j-1})+\Big(w_{h}(x_{i-1},y_{j})-w_{h}(x_{i-1},y_{j-1})\Big)\eta   \nonumber\\
                  &                           +\Big(w_{h}(x_{i},y_{j-1})-w_{h}(x_{i-1},y_{j-1})\Big)\xi.
\end{align}
From this definition, one immediately has
\begin{align}
\label{eq:w1_pointy}
\left\{
\begin{array}{l}
(\Pi_{K}^{1}w_{h})|^{(x_{i},y)}_{(x_{i-1},y)} = w_{h}(x_{i},y_{j-1})-w_{h}(x_{i-1},y_{j-1}),
\quad \forall y\in [y_{j-1},y_j],\\
(\Pi_{K}^{1}w_{h})|^{(x,y_{j})}_{(x,y_{j-1})} = w_{h}(x_{i-1},y_{j})-w_{h}(x_{i-1},y_{j-1}),
\quad \forall x\in [x_{i-1},x_i],
\end{array}
\right.
\end{align}
and
\begin{align}\label{eq:w1_norm}
            \|w_{h}-\Pi_{K}^{1}w_{h}\|_{0,K}
\lesssim &  \,h\,|w_{h}|_{1,K}.
\end{align}

$\bullet$ $\Pi_{h}^{k,*}: L^{2}(\Omega)\rightarrow V_{h}$, the piecewise constant operator based on the dual mesh $\mathcal{T}_{h}^{*}$. Denote $a^x_{s}$, $a^y_{t}$ ($s,t\in\mathbb{Z}_{k}^{0}$) the interpolation parameters on the reference element $\hat{K}$ that
\begin{align*}
&-1=a^x_{0}<a^x_{1}<\cdots<a^x_{k}=1,\\
&-1=a^y_{0}<a^y_{1}<\cdots<a^y_{k}=1,
\end{align*}
which locate the interpolation points $\boldsymbol{\hat{a}_{s,t}} = (a^x_{s},a^y_{t})$ on $\hat{K}$, and the interpolation points on $K\in\mathcal{T}_{h}$
\begin{align}\label{eq:Interpolation nodes}
\boldsymbol{a_{s,t}^{K}} := (a^{x}_{s,K},a^{y}_{t,K}) = F_{K}(\boldsymbol{\hat{a}_{s,t}}),
   \quad \forall s,t \in \mathbb{Z}^{0}_{k}.
\end{align}
Then, each dual element $K_{\boldsymbol{a}}^*\in\mathcal{T}_{h}^{*}$ surrounds one $\boldsymbol{a}:=(a^x,a^y)$ correspondingly. Thus, 
\begin{align*}
(\Pi_{h}^{k,*}w)|_{K_{\boldsymbol{a}}^{*}} = w(\boldsymbol{a}),\quad \forall w\in L^{2}(\Omega), \,\,\forall K_{\boldsymbol{a}}^{*}\in \mathcal{T}_{h}^{*}.
\end{align*}

The operator $\Pi^{k,*}_h$ can be defined through composition. Specifically, for any $w\in L^{2}(\Omega)$,
\begin{align*}
\Pi^{k,*}_{h}w=\Pi^{k,*}_{h,x}(\Pi^{k,*}_{h,y}w)=\Pi^{k,*}_{h,y}(\Pi^{k,*}_{h,x}w),
\end{align*}
where the component operators $\Pi^{k,*}_{h,x}$ and $\Pi^{k,*}_{h,y}$ are defined as
\begin{align*}
(\Pi^{k,*}_{h,x}w)|_{K_{\boldsymbol{a}}^{*}}=w(a^x,y) \quad {\mathrm{and}} \quad
(\Pi^{k,*}_{h,y}w)|_{K_{\boldsymbol{a}}^{*}}=w(x,a^y).
\end{align*}

\begin{remark}
The trial-to-test interpolation operator $\Pi_{h}^{k,*}$ is defined by two aspects: the dual mesh (dual parameters) and interpolation point positions (interpolation parameters). We emphasize that $\Pi_{h}^{k,*}$ is purely theoretical, so the interpolation parameters have no impact on FVE schemes.
\end{remark}
\begin{remark}
The properties (\ref{eq:w1_pointy}) of the interpolation operator $\Pi_{K}^{1}$ ensure that $\Pi_{h}^{k,*}(\Pi_{K}^{1}w_{h}(x_{i},y)    - \Pi_{K}^{1}w_{h}(x_{i-1},y))$ remains a constant. This property plays an important role in the subsequent analysis (\ref{eq:E11_U}) of Theorem~\ref{thm:Ultra-super_H1 2D}.
\end{remark}

\section{Orthogonality condition and asymmetric-enabled M-decompositions}
\label{sec:OC_AMD}
This section provides the definitions and properties of the tensorial orthogonality condition and the asymmetric-enabled M-decompositions. The superconvergence and ultraconvergence properties discussed in this paper are fundamentally determined by the order of the tensorial orthogonality condition satisfied by the FVE scheme. While, the asymmetric-enabled M-decompositions provides superclose functions $u_{I,Super}$ and $u_{I,Ultra}$, which bridge the FVEM solutions with the exact solution at superconvergence points and ultraconvergence points.

\subsection{Tensorial orthogonality condition}
\label{subsubsec:The tensorial orthogonality condition}
First, we recall the 1D $k$-$r$-order ($k-1 \leq r \leq 2k-2$) orthogonality condition \cite{Wang.2024}.
\begin{definition}[One-dimensional $k$-$r$-order orthogonality condition \cite{Wang.2024}]
\label{def:orthogonality_condition 1D}
A trial-to-test operator $\Pi_{h,x}^{k,*}$ satisfies the $k$-$r$-order orthogonality condition if
\begin{align}\label{eq:orthogonality_condition 1D}
\int_{-1}^1 g\,(w-\Pi_{h,x}^{k,*}w)\ud x=0,\quad \forall g \in P^r([-1,1]),\,\forall w \in P^1([-1,1]).
\end{align}
\end{definition}

\begin{remark}
The 1D orthogonality condition $(\ref{eq:orthogonality_condition 1D})$ imposes constraints on the dual parameters $\alpha_s^{x}$ ($s\in\mathbb{Z}_{k}$) through equations:
\begin{align}\label{eq:OCequiv}
\sum\limits_{s=1}^{k}(a_{s}^{x}-a_{s-1}^{x})(\alpha_{s}^{x})^{i+1}
=\frac{1}{i+2}\Big(1-(-1)^i\Big),\quad \forall i\in\{0,1,\dots,r\}.
\end{align}
Notably, the interpolation parameters $a_s^{x}$ ($s\in\mathbb{Z}_{k}\setminus\{k\}$) remain as free parameters that do not affect the FVE scheme or dual mesh construction.
\end{remark}

The tensorial extension to 2D is defined through operator composition.

\begin{definition}[Tensorial $k$-$r$-order orthogonality condition]
\label{def:orthogonality_condition_2D}
A trial-to-test operator $\Pi_{h}^{k,*}$ satisfies the tensorial $k$-$r$-order orthogonality condition if it admits a decomposition
\[
\Pi_h^{k,*} = \Pi_{h,x}^{k,*}  \Pi_{h,y}^{k,*},
\]
and the operators $\Pi_{h,x}^{k,*}$ and $\Pi_{h,y}^{k,*}$ satisfy the $k$-$r$-order orthogonality condition (\ref{eq:orthogonality_condition 1D}) in $x$-direction and $y$-direction, respectively.

A bi-$k$-order FVE scheme is called to satisfy the tensorial $k$-$r$-order orthogonality condition if there exists a trial-to-test operator $\Pi_{h}^{k,*}$ satisfying the above decomposition property.
\end{definition}

The tensorial $k$-$r$-order $(k-1 \le r \le 2k-2)$ orthogonality condition implies the properties in Lemma~\ref{lem:orth1}-\ref{lem:The equivalence of the orth condition}. In \cite{Wang.2024,Zhang.2023}, the case of $r = k-1$ is rigorously proved, and the results for $k \le r \le 2k-2$ follow by a similar argument.

\begin{lemma}
\label{lem:orth1}
Given a rectangular mesh $\mathcal{T}_h$, let $\Pi_h^{k,*}$ satisfy the tensorial $k$-$r$-order orthogonality condition. For any $K\in \mathcal{T}_h$, we have
\begin{subequations}
\begin{align}
\label{eq:orth_conditionx}
&\int_{K} \left( g_1(y)x^i \right) (x-\Pi_h^{k,*}x)\ud x \ud y=0,
    & \forall\, i\le r, \textrm{and polynomial $g_1(y)$ of $y$,}\\
\label{eq:orth_conditiony}
&\int_{K} \left( g_2(x)y^i \right) (y-\Pi_h^{k,*}y)\ud x \ud y=0,
    & \forall\, i\le r, \textrm{and polynomial $g_2(x)$ of $x$,}\\
\label{eq:orth_condition1}
&\int_{K} g_3\,(w-\Pi_h^{k,*}w)\ud x \ud y=0, &\forall\, g_3\in Q^{r}(K),\,  \forall w\in P^1(K),\\
\label{eq:orth_condition2}
&\int_{S_j} g_4\,(w-\Pi_h^{k,*}w)\ud s=0, &\forall\, g_4\in P^{r}(S_j),\,  \forall w\in P^1(S_j),
\end{align}
\end{subequations}
where, $S_j\subset \partial K\,(j=1,\cdots,4)$ are the $4$ edges of $K$. Here, $Q^{r}(\cdot)$ and $P^{r}(\cdot)$ are the bi-$r$-order and $r$-order polynomial spaces respectively.
\end{lemma}

\begin{lemma}
\label{lem:The equivalence of the orth condition}
Let $\Pi_h^{k,*}$ satisfy the tensorial $k$-$r$-order orthogonality condition. Denote $A^{x}_i$, $A^{y}_j$ be the weights of the $k$-point numerical quadrature corresponding to the dual parameters $\alpha^x_{i}$ and $\alpha^y_{j}$ $(i,j\in\mathbb{Z}_{k})$, respectively. Then,
$$
\int_{-1}^{1} g(x) \, dx = \sum_{i=1}^{k} A^{x}_{i} g(\alpha^x_{i})  \text{  and  }  \int_{-1}^{1} g(y) \, dy = \sum_{j=1}^{k} A^{y}_{j} g(\alpha^y_{j}),\quad  \forall g \in P^{r+1}([-1,1]).
$$
\end{lemma}

\subsection{Asymmetric-enabled M-decompositions}
\label{subsec:Ultra M-Decomposition}
Superclose functions $u_{I,Super}$ (superconvergence) and $u_{I,Ultra}$ (ultraconvergence) are defined via asymmetric-enabled M-decompositions AMD-Super and AMD-Ultra, respectively. They are expressed in terms of M-functions, with specific constraints applied to the coefficients in the expansion of the error between these functions and the exact solution.

The 1D \textbf{M-functions} on the 1D reference element $[-1,1]$ are defined by 
\begin{align}\label{eq:M-functions}
\hat{M}_0=1,\,\hat{M}_1=\hat{x},\,\hat{M}_2=\frac{1}{2}(\hat{x}^2-1),\,\dots,\,
\hat{M}_{i+1}=\frac{1}{2^i\,i!}\frac{\ud^{i-1}}{\ud\hat{x}^{i-1}}(\hat{x}^2-1)^i,\,\dots,\,
\end{align}
which have the following properties
\begin{align}\label{eq:M-functions properties}
\left\{ \begin{array}{ll}
\hat{M}_i(\pm 1)=0, & i=2,3,\dots, \\
(\hat{M}_i,\hat{M}_j)=0, & |i-j|\ne 2\quad\!\!\! \mathrm{and}\quad\!\!\! i\ne j.
\end{array}
\right.
\end{align}
Through an affine mapping, the 1D M-functions $M_i(x)$ on any 1D segment element can be defined. For any $K\in \mathcal{T}_{h}$, denote $M_{s}^{x}:=M_{s}(x)$ and $M_{t}^{y}:=M_{t}(y)$ on $K$ \cite{Chen.1995}. Then, decompose $u$ as
\begin{align}   \label{eq:u_Mdecomposition}
u=\left\{
\begin{array}{ll}
\sum\limits_{0\leq s,t\leq k+1} b_{s,t}M^{x}_{s}M^{y}_{t} +O(h^{k+2}), & \text{for AMD-Super,}   \\
\sum\limits_{0\leq s,t\leq k+2} b_{s,t}M^{x}_{s}M^{y}_{t} +O(h^{k+3}), & \text{for AMD-Ultra.}
\end{array}
\right.
\end{align}

\begin{lemma}\label{lem:M_function_Coeff}
Given a regular rectangular mesh $\mathcal{T}_{h}$, we have the following estimate of the coefficients of the M-decomposition (\ref{eq:u_Mdecomposition}) in adjoint elements.
\begin{subequations}
\begin{align}
\label{eq:M_function_x_Coeff}
&\frac{b_{1,k+1}^{K_{i,j}}}{h^{x}_{i}}-\frac{b_{1,k+1}^{K_{i+1,j}}}{h^{x}_{i+1}} = \mathnormal{O}(h^{k+2}),
       \quad \forall i\in\mathbb{Z}_{N_x-1}, j\in\mathbb{Z}_{N_y},\\
\label{eq:M_function_y_Coeff}
&\frac{b_{k+1,1}^{K_{i,j}}}{h^{y}_{j}}-\frac{b_{k+1,1}^{K_{i,j+1}}}{h^{y}_{j+1}} = \mathnormal{O}(h^{k+2}).
       \quad \forall i\in\mathbb{Z}_{N_x}, j\in\mathbb{Z}_{N_y-1}.
\end{align}
\end{subequations}
\end{lemma}

\begin{proof}
By the definition of $b_{s,t}^{K_{i,j}}$ on $K_{i,j}$ in (\ref{eq:u_Mdecomposition}) and integration by parts in $x$-direction, one has
\begin{align*}
b_{1,k+1}^{K_{i,j}} =& \frac{h^{y}_{j}}{2} \frac{2k+1}{4}\int_{-1}^{1} \frac{\partial ( u(x_{i},y)-u(x_{i-1},y) )}{\partial y}  \, L_{k}(\hat{y}) \ud \hat{y}\\
                    =& \frac{h^{x}_{i}}{2} \frac{h^{y}_{j}}{2} \frac{2k+1}{4}\iint_{\hat{K}}
                         \frac{\partial^2 u }{\partial x\partial y}(x,y)  \, L_{k}(\hat{y})\ud \hat{x}\ud \hat{y}\\
                    =& \frac{h^{x}_{i}}{2}\Big(\frac{h^{y}_{j}}{2}\Big)^{k+1}\frac{(2k+1)(-1)^{k}}{2^{k+2} k!}
                         \iint_{\hat{K}}  \frac{\partial^{k+2} u }{\partial x \partial y^{k+1}}(x,y)
                         (\hat{y}^{2}-1)^{k}\ud \hat{x}\ud \hat{y},
\end{align*}
where $\displaystyle x = \frac{h^{x}_{i}\hat{x}+x_{i-1}+x_{i}}{2}$, $\displaystyle y = \frac{h^{y}_{j}\hat{y}+y_{j-1}+y_{j}}{2}$, and $\displaystyle L_{k}(\hat{y})=\frac{1}{2^k\,k!}\frac{\ud^{k}}{\ud\hat{y}^{k}}(\hat{y}^2-1)^k$ is the Legendre function. Then, by the mean value theorem, we have
\begin{align*}
 \frac{b_{1,k+1}^{K_{i,j}}}{h^{x}_{i}}-\frac{b_{1,k+1}^{K_{i+1,j}}}{h^{x}_{i+1}}
=& \Big(\frac{h^{y}_{j}}{2}\Big)^{k+1}\frac{(2k+1)(-1)^{k}}{2^{k+3} k!}
     \iint_{\hat{K}}
     \Big(\frac{\partial^{k+2} u(\frac{h^{x}_{i}\hat{x}+x_{i-1}+x_{i}}{2},y)}{\partial x \partial y^{k+1}}\nonumber\\
 &\qquad
    -\frac{\partial^{k+2} u (\frac{h^{x}_{i+1}\hat{x}+x_{i}+x_{i+1}}{2},y )}{\partial x \partial y^{k+1}}\Big)
      (\hat{y}^{2}-1)^{k}\ud \hat{x}\ud \hat{y}\nonumber\\
=& \Big(\frac{h^{y}_{j}}{2}\Big)^{k+1}\frac{(2k+1)(-1)^{k}}{2^{k+3} k!}\iint_{\hat{K}}
      H_{\hat{x}}\frac{\partial^{k+3} u (\theta(\hat{x}),y)}{\partial x^{2} \partial y^{k+1}}
      (\hat{y}^{2}-1)^{k}\ud \hat{x}\ud \hat{y},
\end{align*}
where $\displaystyle \theta(\hat{x}) \in (\frac{h^{x}_{i}\hat{x}+x_{i-1}+x_{i}}{2},\, \frac{h^{x}_{i+1}\hat{x}+x_{i}+x_{i+1}}{2})$ and
\begin{align*}
H_{\hat{x}} =  \frac{(h^{x}_{i}-h^{x}_{i+1})\hat{x}-(h^{x}_{i}+h^{x}_{i+1})}{2}=O(h).
\end{align*}
Thus, we have the estimate (\ref{eq:M_function_x_Coeff}). Similarly, we have (\ref{eq:M_function_y_Coeff}).
\end{proof}
\begin{remark}
The higher-order approximation of adjacent elements in Lemma~\ref{lem:M_function_Coeff} plays a crucial role in the proof of Theorem~\ref{thm:Ultra-super_H1 2D} (\ref{eq:elr1U}). Notably, these properties does not require adjacent elements to have identical sizes and hold for quasi-uniform rectangular meshes.
\end{remark}

\begin{table}[htbp!]
  \centering
  \caption{Structure of AMD-Super for the case $k=3$.}\label{tab:relations1}
    \begin{minipage}[t]{.43\textwidth}
    \centering
    \title{a. Expression of $u_{I,Super}$}
    \vspace{0.1cm}
    \begin{tabular}{c|cc:ccc}
          & $1$       & $M_1^{y}$ & $M_2^{y}$ &  $M_3^{y}$ &  $M_4^{y}$\\
    \hline
    1     & $\boldsymbol{b_{00}^{S}}$  & $\boldsymbol{b_{01}^{S}}$  & $\boldsymbol{b_{02}^{S}}$  & $\boldsymbol{b_{03}^{S}}$  & -  \\
    $M_1^{x}$     & $\boldsymbol{b_{10}^{S}}$  & $\boldsymbol{b_{11}^{S}}$  & $\boldsymbol{b_{12}^{S}}$  & $\boldsymbol{b_{13}^{S}}$  & -  \\
    \hdashline
    $M_2^{x}$     & $\boldsymbol{b_{20}^{S}}$  & $\boldsymbol{b_{21}^{S}}$  & $\boldsymbol{b_{22}^{S}}$  & -  & -  \\
    $M_3^{x}$     & $\boldsymbol{b_{30}^{S}}$  & $\boldsymbol{b_{31}^{S}}$  & -  & -  & -  \\
    $M_4^{x}$     & -  & -  & -  & -  & -  \\
    \end{tabular}
    \end{minipage}
  \hspace{.2in}
    \begin{minipage}[t]{.43\textwidth}
    \centering
    \title{b. Expression of $u-u_{I,Super}$}
    \vspace{0.1cm}
    \begin{tabular}{c|cc:ccc}
          & $1$       & $M_1^{y}$ & $M_2^{y}$ &  $M_3^{y}$ &  $M_4^{y}$\\
    \hline
    1     & -  & -  & $\boldsymbol{b_{02}^{*}}$  & $\boldsymbol{b_{03}^{*}}$  & $b_{04}$  \\
    $M_1^{x}$     & -  & -  & $\boldsymbol{b_{12}^{*}}$  & $\boldsymbol{b_{13}^{*}}$  & $b_{14}$  \\
    \hdashline
    $M_2^{x}$     & $\boldsymbol{b_{20}^{*}}$  & $\boldsymbol{b_{21}^{*}}$  & -  & $b_{23}$  & $b_{24}$  \\
    $M_3^{x}$     &  $\boldsymbol{b_{30}^{*}}$  & $\boldsymbol{b_{31}^{*}}$  & $b_{32}$  & $b_{33}$  & $b_{34}$  \\
    $M_4^{x}$     & $b_{40}$  & $b_{41}$  & $b_{42}$  & $b_{43}$  & $b_{44}$  \\
    \end{tabular}
    \end{minipage}
\end{table}

\subsubsection{AMD-Super}
\begin{definition}[The AMD-Super constraints]\label{def:SMD 2D}
For any $K\in\mathcal{T}_{h}$, suppose $u\in H^{k+2}(\Omega)$. Let $\alpha_{m}^{x}$ and $\alpha_{m}^{y}$ ($m\in\mathbb{Z}_{k}$) be the dual parameters. Then, define the superclose function $u_{I, Super}\in U_{h}^{k}$ with expression (see Table~\ref{tab:relations1} for the case of $k=3$ as an example)
\begin{align}
&u_{I,Super} = \sum_{s + t \leq k + 1, \,\, s, t \leq k} b^{S}_{s,t}M^{x}_{s}M^{y}_{t},     \label{eq:UI_Super}\\
&b^{S}_{s,t}=\left\{
\begin{array}{ll}
b_{s,t},              &  (s,t)\in     \{0\leq s,t\leq 1\}\cup\{s,t\ge 2, \,\,  s+t\leq k+1\},\\
b_{s,t}-b^{*}_{s,t},  &  (s,t) \in    \{s=0,1, \,\, 2\leq t\leq k\}\cup\{2\leq s\leq k, \,\, t=0,1\},
\end{array}
\right.         \nonumber
\end{align}
with $b_{s,t}$ being the coefficients of $u$ in (\ref{eq:u_Mdecomposition}) and $b^{*}_{s,t}$ being the corrections defined by the following AMD-Super constraints
\begin{subequations}
\begin{align}
\label{eq:SMD2Da}
\sum^{k}_{s=2}b^{*}_{s,t}\hat{M}'_{s}(\alpha_m^x)+b_{k+1,t}\hat{M}'_{k+1}(\alpha_m^x)=0,&
\quad t=0,1\quad\!\!\! \mathrm{and}\quad\!\!\!m\in \mathbb{Z}_{k-1},\\
\label{eq:SMD2Dd}
\sum^{k}_{t=2}b^{*}_{s,t}\hat{M}'_{t}(\alpha_m^y)+b_{s,k+1}\hat{M}'_{k+1}(\alpha_m^y)=0,&
\quad s=0,1\quad\!\!\! \mathrm{and}\quad\!\!\!m\in \mathbb{Z}_{k-1}.
\end{align}
\end{subequations}
\end{definition}

The difference between $u$ and $u_{I,Super}$ on $K\in\mathcal{T}_h$ can be expressed as
\begin{align}\label{eq:u-uI res_S}
(u-u_{I,Super})|_{K} = R^{\mathrm{S}}_{x,K}+R^{\mathrm{S}}_{y,K}+R^{\mathrm{S},1}_{x,K}M^y_{1}+R^{\mathrm{S},1}_{y,K} M^x_{1}+  R^{\mathrm{S}}_{res},
\end{align}
where
\begin{align*}
&R^{\mathrm{S}}_{x,K}        = \sum^{k}_{s=2}b^{*}_{s,0}M^x_{s}+b_{k+1,0}M^x_{k+1},\qquad\qquad
 R^{\mathrm{S}}_{y,K}        = \sum^{k}_{t=2}b^{*}_{0,t}M^y_{t}+b_{0,k+1}M^y_{k+1},\\
&
R^{\mathrm{S},1}_{x,K} = \sum^{k}_{s=2} b^{*}_{s,1}M^x_{s}  +  b_{k+1,1}M^x_{k+1}  =O(h^{k+2}), \\
&
R^{\mathrm{S},1}_{y,K} = \sum^{k}_{t=2} b^{*}_{1,t}M^y_{t}  +  b_{1,k+1}M^y_{k+1}=O(h^{k+2}),\quad R^{\mathrm{S}}_{res}=O(h^{k+2}).
\end{align*}
When the dual strategy satisfies the tensorial $k$-$(k-1)$-order orthogonality condition, noticing that $\displaystyle \frac{\partial R^{\mathrm{S}}_{x,K}}{\partial x}\in P^{k}$ and $\displaystyle \frac{\partial R^{\mathrm{S}}_{y,K}}{\partial y}\in P^{k}$, combining with Lemma~\ref{lem:The equivalence of the orth condition} and the properties (\ref{eq:M-functions properties}) of the M-functions, (\ref{eq:SMD2Da})-(\ref{eq:SMD2Dd}) are also valid for $m=k$.
\begin{align}
\label{eq:Rx_alpha_S}
  \frac{\partial R^{\mathrm{S}}_{x,K}}{\partial x} (\alpha_{m,K}^x, y)=0 \,\, \mathrm{and} \,\,  \frac{\partial R^{\mathrm{S}}_{y,K}}{\partial y} (x,\alpha_{m,K}^y)=0, \quad m\in\mathbb{Z}_{k},   \\
\label{eq:M-functions_xpointS}
  R^{\mathrm{S}}_{x,K}(x_{i-1},y) = R^{\mathrm{S}}_{x,K}(x_{i},y) = 0 \,\, \mathrm{and}\,\,   R^{\mathrm{S}}_{y,K}(x,y_{j-1}) = R^{\mathrm{S}}_{y,K}(x,y_{j}) = 0.
\end{align}
Denote the set of the $(k+1)$ roots for $R_{K,\mathrm{S}}^{x}$ and $R_{K,\mathrm{S}}^{y}$ (as a 1D function of $x$ or $y$) as $\mathbb{P}_{x,K}^{\mathrm{S}}$ and $\mathbb{P}_{y,K}^{\mathrm{S}}$. Then, $u_{I,Super}$ has following superclose properties.
\begin{lemma}[Superclose properties of $u_{I,Super}$]\label{lem:u_uI_normS}
Given a regular rectangular mesh $\mathcal{T}_{h}$ of $\Omega$, let $u\in H^{1}_{0}(\Omega)\cap H^{k+2}(\Omega)$ and $u_{I,Super}\in U^{k}_h$ be derived by Definition~\ref{def:SMD 2D}. Then, the \textbf{global approximation properties} for $u_{I,Super}$ hold
\begin{subequations}
\begin{align}
\| u-u_{I,Super} \|_{1}\lesssim& \,h^{k} \|u\|_{k+1},     \label{eq:u_uI_H1}\\
\| u-u_{I,Super} \|_{0}\lesssim& \,h^{k+1} \|u\|_{k+1}.       \label{eq:u_uI_L2}
\end{align}
\end{subequations}
Furthermore, for any $K\in\mathcal{T}_{h}$, when the dual strategy satisfies the tensorial $k$-$(k-1)$-order orthogonality condition, the function $u_{I,Super}$ exhibits \textbf{derivative superclose properties}
\begin{subequations}
\begin{align}
    \label{eq:u_uI_H1_xSdisc}
 \Big|\frac{\partial }{\partial x} (u - u_{I,Super})\Big|
     = \mathnormal{O}(h^{k+1}), \quad \forall x\in\{\alpha_{m,K}^x\}_{m\in\mathbb{Z}_k},\,\, \forall y\in [y_{j-1},y_{j}], \\
    \label{eq:u_uI_H1_ySdisc}
 \Big|\frac{\partial }{\partial y} (u - u_{I,Super})\Big|
    = \mathnormal{O}(h^{k+1}), \quad \forall x\in [x_{i-1},x_{i}],\,\, \forall y\in \{ \alpha_{m,K}^y \}_{m\in\mathbb{Z}_k}.
\end{align}
\end{subequations}
When the dual strategy satisfies the tensorial $k$-$k$-order orthogonality condition, there is \textbf{function-value superclose properties}
\begin{align}\label{eq:u_uI_L2_Sdisc}
      |u - u_{I,Super}| = \mathnormal{O}(h^{k+2}),\quad \forall x\in\mathbb{P}_{x,K}^{\mathrm{S}},\,\, \forall y\in \mathbb{P}_{y,K}^{\mathrm{S}}.
\end{align}
\end{lemma}

\subsubsection{AMD-Ultra}
\begin{definition}[The AMD-Ultra constraints]\label{def:UMD 2D}
For any $K\in\mathcal{T}_{h}$, suppose $u\in H^{k+3}(\Omega)$. Let $\alpha_{m}^{x}$ and $\alpha_{m}^{y}$ ($m\in\mathbb{Z}_{k}$) be the dual parameters. Then, define the superclose function $u_{I, Ultra}\in U_{h}^{k}$ with expression (see Table~\ref{tab:relations2} for the case of $k=3$ as an example)
\begin{align*}
&u_{I,Ultra} = \sum_{s + t \leq k + 2, \,\, s, t \leq k} b^{U}_{s,t}M^{x}_{s}M^{y}_{t},\\
&b^{U}_{s,t}=\left\{
\begin{array}{ll}
b_{s,t},                      &  (s,t)\in     \{0\leq s,t\leq 1\}\cup\{s,t\ge 2 , \,\, s+t\leq k+2\},\\
b_{s,t}-b^{*}_{s,t}-b^{1,*}_{s,t},  &  (s,t) \in    \{s=0,1, \,\, 2\leq t\leq k\}\cup\{2\leq s\leq k,\,\, t=0,1\},
\end{array}
\right.
\end{align*}
with $b^{*}_{s,t}$ satisfying (\ref{eq:SMD2Da})-(\ref{eq:SMD2Dd}) and $b^{1,*}_{s,t}$ satisfying
\begin{subequations}
\begin{align}
\label{eq:UMD2Da}
\sum^{k}_{s=2}b^{1,*}_{s,t}\hat{M}'_{s}(\alpha_m^x)    +   b_{k+2,t}\hat{M}'_{k+2}(\alpha_m^x)=0,&
\quad t=0,1\quad\!\!\! \mathrm{and}\quad\!\!\!m\in \mathbb{Z}_{k-1},\\
\label{eq:UMD2Dd}
\sum^{k}_{t=2}b^{1,*}_{s,t}\hat{M}'_{t}(\alpha_m^y)    +   b_{s,k+2}\hat{M}'_{k+2}(\alpha_m^y)=0,&
\quad s=0,1\quad\!\!\! \mathrm{and}\quad\!\!\!m\in \mathbb{Z}_{k-1}.
\end{align}
\end{subequations}
The equations (\ref{eq:SMD2Da})-(\ref{eq:SMD2Dd}) and (\ref{eq:UMD2Da})-(\ref{eq:UMD2Dd}) together are called the AMD-Ultra constraints.
\end{definition}

\begin{table}[htbp!]
  \centering
  \caption{Structure of the AMD-Ultra for the case $k=3$.}\label{tab:relations2}

    \begin{minipage}[t]{.41\textwidth}
    \centering
    \title{a. Expression of $u_{I,Ultra}$}
    \vspace{0.1cm}

  \resizebox{\textwidth}{!}
  {    \begin{tabular}{c|cc:ccc}
          & $1$       & $M_1^{y}$ & $M_2^{y}$ &  $M_3^{y}$ &  $M_4^{y}$\\
    \hline
    1     & $\boldsymbol{b_{00}^{U}}$  & $\boldsymbol{b_{01}^{U}}$  & $\boldsymbol{b_{02}^{U}}$  & $\boldsymbol{b_{03}^{U}}$  & -  \\
    $M_1^{x}$     & $\boldsymbol{b_{10}^{U}}$  & $\boldsymbol{b_{11}^{U}}$  & $\boldsymbol{b_{12}^{U}}$  & $\boldsymbol{b_{13}^{U}}$  & -  \\
    \hdashline
    $M_2^{x}$     & $\boldsymbol{b_{20}^{U}}$  & $\boldsymbol{b_{21}^{U}}$  & $\boldsymbol{b_{22}^{U}}$  & $\boldsymbol{b_{23}^{U}}$  & -  \\
    $M_3^{x}$     & $\boldsymbol{b_{30}^{U}}$  & $\boldsymbol{b_{31}^{U}}$  & $\boldsymbol{b_{32}^{U}}$  & -  & -  \\
    $M_4^{x}$     & -  & -  & -  & -  & -  \\
    \end{tabular}
  }
    \end{minipage}
  \hspace{.2in}
    \begin{minipage}[t]{.46\textwidth}
    \centering
  \title{b. Expression of $b_{st}^{1,*}$ in $u-u_{I,Ultra}$}
    \vspace{0.1cm}

  \resizebox{\textwidth}{!}
  {
     \begin{tabular}{c|cc:cccc}
          & $1$       & $M_1^{y}$ & $M_2^{y}$ &  $M_3^{y}$ &  $M_4^{y}$ &  $M_5^{y}$\\
    \hline
    1     & -  & -  & $\boldsymbol{b_{02}^{1,*}}$  & $\boldsymbol{b_{03}^{1,*}}$  & $b_{04}$   & $b_{05}$  \\
    $M_1^{x}$     & -  & -  & $\boldsymbol{b_{12}^{1,*}}$  & $\boldsymbol{b_{13}^{1,*}}$  & $b_{14}$   & $b_{15}$  \\
    \hdashline
    $M_2^{x}$     & $\boldsymbol{b_{20}^{1,*}}$  & $\boldsymbol{b_{21}^{1,*}}$  & -  & -  & $b_{24}$  & $b_{25}$  \\
    $M_3^{x}$     &  $\boldsymbol{b_{30}^{1,*}}$  & $\boldsymbol{b_{31}^{1,*}}$  & -  & $b_{33}$  & $b_{34}$  & $b_{35}$  \\
    $M_4^{x}$     & $b_{40}$  & $b_{41}$  & $b_{42}$  & $b_{43}$  & $b_{44}$  & $b_{45}$  \\
    $M_5^{x}$     & $b_{50}$  & $b_{51}$  & $b_{52}$  & $b_{53}$  & $b_{54}$  & $b_{55}$  \\
    \end{tabular}
  }
    \end{minipage}
\end{table}

The difference between $u$ and $u_{I,Ultra}$ on $K \in \mathcal{T}_{h}$ can be expressed as
\begin{align}\label{eq:u-uI_res_U}
u-u_{I,Ultra} = &  R^{\mathrm{S}}_{x,K}+R^{\mathrm{U}}_{x,K}+ R^{\mathrm{S}}_{y,K} + R^{\mathrm{U}}_{y,K} \nonumber\\
                & +(R^{\mathrm{S},1}_{x,K}+R^{\mathrm{U},1}_{x,K})M^y_{1}+(R^{\mathrm{S},1}_{y,K}+R^{\mathrm{U},1}_{y,K})M^x_{1}+  R^{\mathrm{U}}_{res},
\end{align}
where, $R^{\mathrm{S}}_{x,K}$, $R^{\mathrm{S}}_{y,K}$, $R^{\mathrm{S},1}_{x,K}$, and $R^{\mathrm{S},1}_{y,K}$ are given by (\ref{eq:u-uI res_S}), and
\begin{align*}
&R^{\mathrm{U}}_{x,K}        = \sum^{k}_{s=2}b^{1,*}_{s,0}M^x_{s}+b_{k+2,0}M^x_{k+2}=O(h^{k+2}),\\
& R^{\mathrm{U}}_{y,K}        = \sum^{k}_{t=2}b^{1,*}_{0,t}M^y_{t}+b_{0,k+2}M^y_{k+2}=O(h^{k+2}),\\
&
R^{\mathrm{U},1}_{x,K} =  \sum^{k}_{s=2} b^{1,*}_{s,1}M^x_{s}  +  b_{k+2,1}M^x_{k+2} =O(h^{k+3}),   \\
&
R^{\mathrm{U},1}_{y,K} =  \sum^{k}_{t=2} b^{1,*}_{1,t}M^y_{t}  +  b_{1,k+2}M^y_{k+2} =O(h^{k+3}),   \qquad  R^{\mathrm{U}}_{res}=O(h^{k+3}).
\end{align*}

When the dual strategy satisfies the tensorial $k$-$k$-order orthogonality condition, noticing that $\displaystyle \frac{\partial R^{\mathrm{U}}_{x,K}}{\partial x},\,\frac{\partial R^{\mathrm{S,1}}_{x,K}}{\partial x}\in P^{k+1}$ and $\displaystyle \frac{\partial R^{\mathrm{U}}_{y,K}}{\partial y},\,\frac{\partial R^{\mathrm{S,1}}_{y,K}}{\partial y}\in P^{k+1}$, combining with Lemma~\ref{lem:The equivalence of the orth condition} and the properties (\ref{eq:M-functions properties}) of the M-functions, (\ref{eq:UMD2Da})-(\ref{eq:UMD2Dd}) are also valid for $m=k$. Thus, for all $m\in\mathbb{Z}_{k}$,
\begin{align}
\label{eq:Rx_alpha_U}
&\left\{
\begin{array}{ll}
\displaystyle \frac{\partial R^{\mathrm{U}}_{x,K}}{\partial x} (\alpha_{m,K}^x, y)=0,           &   \mathrm{and} \,\,  \displaystyle \frac{\partial R^{\mathrm{U}}_{y,K}}{\partial y} (x,\alpha_{m,K}^y)=0,   \\
\displaystyle \frac{\partial (R^{\mathrm{S,1}}_{x,K}M_1^{y})}{\partial x} (\alpha_{m,K}^x, y)=0, ,        &   \mathrm{and} \,\,   \displaystyle \frac{\partial (R^{\mathrm{S,1}}_{x,K}M_1^{y})}{\partial y} (x, y)=0,  \,\,  \forall x\in\mathbb{P}_{x,K}^{\mathrm{S}},\\
\displaystyle \frac{\partial (R^{\mathrm{S,1}}_{y,K}M_1^{x})}{\partial x} (x, y)=0,   \,\,  \forall y\in\mathbb{P}_{y,K}^{\mathrm{S}}, &   \mathrm{and} \,\,   \displaystyle \frac{\partial (R^{\mathrm{S,1}}_{y,K}M_1^{x})}{\partial y} (x,\alpha_{m,K}^y)=0,\\
\end{array}
\right. \\
\label{eq:M-functions_xpointU}
&\left\{
\begin{array}{ll}
R^{\mathrm{U}}_{x,K}(x_{i-1},y) = R^{\mathrm{U}}_{x,K}(x_{i},y) = 0         &   \mathrm{and} \quad   R^{\mathrm{U}}_{y,K}(x,y_{j-1}) = R^{\mathrm{U}}_{y,K}(x,y_{j}) = 0,     \\
R^{\mathrm{S},1}_{x,K}(x_{i-1},y) = R^{\mathrm{S},1}_{x,K}(x_{i},y) = 0     &   \mathrm{and} \quad   R^{\mathrm{S},1}_{y,K}(x,y_{j-1}) = R^{\mathrm{S},1}_{y,K}(x,y_{j}) = 0.
\end{array}
\right.
\end{align}
\begin{remark}
Since equations (\ref{eq:SMD2Da}) for $t=0,1$ share same zeros points $\alpha_m^{x}$ ($m\in\mathbb{Z}_{k}$), $R_{x,K}^{\mathrm{S},1}$ is a scalar multiple of $R_{x,K}^{\mathrm{S}}$. Consequently, $\mathbb{P}_{x,K}^{\mathrm{S}}$ also constitutes the root set of $R_{x,K}^{\mathrm{S},1}$. Similarly, $\mathbb{P}_{y,K}^{\mathrm{S}}$ forms the root set of $R_{y,K}^{\mathrm{S},1}$.
\end{remark}

\begin{lemma}[Superclose properties of $u_{I,Ultra}$]\label{lem:u_uI_normU}
Given a regular rectangular mesh $\mathcal{T}_{h}$ of $\Omega$, let $u\in H^{1}_{0}(\Omega)\cap H^{k+3}(\Omega)$ and $u_{I,Ultra}\in U^{k}_h$ be derived by Definition~\ref{def:UMD 2D}. Then, similar \textbf{global approximation properties} as (\ref{eq:u_uI_H1})-(\ref{eq:u_uI_L2}) hold for $u_{I,Ultra}$.

Furthermore, when the dual strategy satisfies the tensorial $k$-$k$-order orthogonality condition, for any $K\in\mathcal{T}_{h}$, the function $u_{I,Ultra}$ exhibits \textbf{derivative superclose properties}
\begin{subequations}
\begin{align}
    \label{eq:u_uI_H1_xUdisc}
 \Big|\frac{\partial }{\partial x} (u - u_{I,Ultra})\Big|
     = \mathnormal{O}(h^{k+2}), \quad x\in \{\alpha_{m,K}^x\}_{m\in\mathbb{Z}_k} \,\,\mathrm{and}\,\,  y\in \mathbb{P}_{y,K}^{\mathrm{S}}, \\
    \label{eq:u_uI_H1_yUdisc}
 \Big|\frac{\partial }{\partial y} (u - u_{I,Ultra})\Big|
    = \mathnormal{O}(h^{k+2}), \quad x\in \mathbb{P}_{x,K}^{\mathrm{S}}\,\,\mathrm{and}\,\,  y\in \{\alpha_{m,K}^y\}_{m\in\mathbb{Z}_k} .
\end{align}
\end{subequations}
\end{lemma}

\begin{lemma}\label{lem:M_function}
When the dual strategy satisfies the tensorial $k$-$k$-order orthogonality condition, we have the following properties for terms in (\ref{eq:u-uI res_S}) and (\ref{eq:u-uI_res_U}).
\begin{subequations}
\begin{align}
\label{eq:M_function_x}
&\int_{x_{i-1}}^{x_{i}} R^{\mathrm{S}}_{x,K} \ud x = \int_{x_{i-1}}^{x_{i}} R^{\mathrm{S},1}_{x,K} \ud x = 0,\quad \forall i\in \mathbb{Z}_{N_x},\\
\label{eq:M_function_y}
&\int_{y_{j-1}}^{y_{j}} R^{\mathrm{S}}_{y,K} \ud y = \int_{y_{j-1}}^{y_{j}} R^{\mathrm{S},1}_{y,K} \ud y = 0,\quad \forall j\in \mathbb{Z}_{N_y}.
\end{align}
\end{subequations}
\end{lemma}
This lemma can be derived with a similar argument of Lemma 7 in \cite{Wang.2024}.

\section{Preliminary estimates}
\label{sec:Preparation_estimation}
This section presents the necessary preliminary estimates beyond the previous lemmas, including the inf-sup condition for the FVE bilinear form (Lemma~\ref{lem:Inf-sup condition}) and the estimate of the difference between the bilinear forms of FEM and FVEM (Lemma~\ref{lem:Estimation_for_E2}).

The inf-sup condition for FVE schemes on 2D quadrilateral meshes is given by \cite{Wang.2025}, which is obviously applicable to rectangular meshes.

\begin{lemma}[Inf-sup condition \cite{Wang.2025}]
 \label{lem:Inf-sup condition}
For sufficiently small mesh size $h$, the following inf-sup condition holds for the bilinear form of the FVE scheme (\ref{eq:FVEscheme}).
\begin{align}
\label{eq:Inf-sup condition}
\inf_{w_{h}\in U_{h}^{k}} \sup_{v_{h}\in U_{h}^{k}} \frac{a_{h}(w_{h},\Pi^{k,*}_{h} v_{h})}{\|w_{h}\|_{1}\, \|v_{h}\|_{1}} \geq c_{0},
\end{align}
where $c_{0} > 0$ is an $h$-independent constant.
\end{lemma}

The analysis of both superconvergence and ultraconvergence requires estimating the difference between the elementwise bilinear forms of the FVEM and FEM. Therefore, in Lemma~\ref{lem:Estimation_for_E2}, we derive this estimate for the general case of equation (\ref{eq:BVP}), which involves variable coefficients for diffusion, convection, and reaction terms, with the diffusion coefficient being a full tensor.

\begin{lemma}\label{lem:Estimation_for_E2}
Given a rectangular mesh $\mathcal{T}_{h}$, let $u\in H^{1}_{0}(\Omega)\cap H^{r+3}(\Omega)$ $(k-1 \le r \le 2k-2)$ be the solution of (\ref{eq:BVP}) and $u_I \in U_{h}^{k}$ satisfy (\ref{eq:u_uI_H1})-(\ref{eq:u_uI_L2}). For $K\in\mathcal{T}_{h}$, when $\Pi_{h}^{k,*}$ satisfies the tensorial $k$-$r$-order orthogonality condition, there holds
\begin{align}\label{eq:Estimation for E2}
\left|  a^{K}_{h}(u-u_{I}, \Pi_{h}^{k,*} w_{1,K})-a^{K}(u-u_I, w_{1,K})  \right|
\le C h^{r+2} \|u\|_{r+3,K}\,|w_{1,K}|_{1,K},
\end{align}
for any $w_{1,K}\in P^{1}(K)$. Here, $a^{K}(\cdot,\cdot)$ is the element-wise bilinear form of FEM
\begin{align} \label{eq:FEM}
a^{K}(v,w)     = \int_{K} (\mathbb{D}\nabla v)\cdot \nabla w + (\mathbb{Q}\cdot\nabla v+\boldsymbol{r} v)w \, \ud x \ud y, \quad \forall v,w \in H^{1}_{0}(\Omega).
\end{align}
\end{lemma}
\begin{proof}
Applying integration by parts to both $a^{K}_{h}(\cdot,\cdot)$ and $a^{K}(\cdot,\cdot)$, their difference can be expressed as
\begin{align}\label{eq:E2}
  a^{K}_{h}(u-u_{I}, \Pi_{h}^{k,*}w_{1,K})-a^{K}(u-u_I, w_{1,K})
= E_{\mathbb{D},1}+E_{\mathbb{D},2}+E_{\mathbb{Q}}+E_{\boldsymbol{r}},
\end{align}
where
\begin{align*}
&E_{\mathbb{D},1} =        -\int_{\partial K} \left(\mathbb{D}\nabla (u-u_{I})\right)\cdot\boldsymbol{n}
                           (w_{1,K}-\Pi_{h}^{k,*}w_{1,K})\ud s,\\
&E_{\mathbb{D},2} =        \int_{K}\nabla\cdot\left(\mathbb{D}\nabla (u-u_{I})\right)
                           (w_{1,K}-\Pi_{h}^{k,*}w_{1,K})\ud x \ud y,\\
&E_{\mathbb{Q}} =        -\int_{K} \left( q_{1}\frac{\partial (u-u_{I})}{\partial x} + q_{2}\frac{\partial (u-u_{I})}{\partial y} \right)
                           (w_{1,K}-\Pi_{h}^{k,*}w_{1,K})\ud x \ud y,\\
&E_{\boldsymbol{r}} =     -\int_{K}\boldsymbol{r}(u-u_{I})
                           (w_{1,K}-\Pi_{h}^{k,*}w_{1,K})\ud x \ud y.
\end{align*}

For $E_{\boldsymbol{r}}$, for the cases of $k-1 \le r \le k$, the Cauchy's inequality together with (\ref{eq:u_uI_L2}) derive
\begin{align}\label{eq:E2R1}
|E_{\boldsymbol{r}}| \le C h^{k+2} \|u\|_{k+1,K} \, |w_{1,K}|_{1,K}.
\end{align}
For the cases of $k+1 \le r \le 2k-2$, inserting $\tilde{\boldsymbol{r}} := \Pi_{h}^{r-k-1}\boldsymbol{r}$, such that $\tilde{\boldsymbol{r}}u_{I} \in Q^{r-1}(K)$. Then, combining with (\ref{eq:orth_condition1}), one has
\begin{align*}
E_{\boldsymbol{r}}=   &-\int_{K}(\boldsymbol{r}-\tilde{\boldsymbol{r}})(u-u_{I})
               (w_{1,K}-\Pi_{h}^{k,*}w_{1,K})\ud x \ud y
             \\
           &
        -\int_{K}\left(\tilde{\boldsymbol{r}}u-\Pi_{h}^{r}(\tilde{\boldsymbol{r}}u)\right)
               (w_{1,K}-\Pi_{h}^{k,*}w_{1,K})\ud x \ud y,
\end{align*}
which together with (\ref{eq:u_uI_L2}) leads to
\begin{align}\label{eq:E2R2}
|E_{\boldsymbol{r}}| \le C h^{r+2} \|u\|_{r+1,K} |w_{1,K}|_{1,K}.
\end{align}

For $E_{\mathbb{Q}}$, for the case of $r = k-1$, the Cauchy's inequality together with (\ref{eq:u_uI_H1}) derive
\begin{align}\label{eq:E2Q1_1}
|E_{\mathbb{Q}}| \le C h^{k+2} \|u\|_{k+1,K} |w_{1,K}|_{1,K}.
\end{align}
For the case $k \leq r \leq 2k - 2$, construct interpolants $\tilde{q}_1 \in Q^{(r-k), (r-k-1)}([x_{i-1}, x_i] \times [y_{j-1}, y_j])$ and $\tilde{q}_2 \in Q^{(r-k-1), (r-k)}([x_{i-1}, x_i] \times [y_{j-1}, y_j])$ such that $\displaystyle\tilde{q}_1 \frac{\partial u_I}{\partial x}, \displaystyle\tilde{q}_2 \frac{\partial u_I}{\partial y} \in Q^r(K)$. Here, $Q^{(r-k), (r-k-1)}$ denotes the bivariate polynomial space of degree at most $(r-k)$ in $x$ and $(r-k-1)$ in $y$. Then, similar arguments with the proof of (\ref{eq:E2R2}) derive
\begin{align}\label{eq:E2Q1_2}
|E_{\mathbb{Q}}| \le C h^{r+2} \|u\|_{r+2,K} |w_{1,K}|_{1,K}.
\end{align}

For the diffusion-related terms $E_{\mathbb{D},1}$ and $E_{\mathbb{D},2}$, we observe cancellation properties. This motivates us to consider these two terms together in the analysis. Denoting $K:=[x_{i-1},\,x_{i}]\times[y_{j-1},\,y_{j}]$ and $P_{i,j}:=(x_{i},y_{j})$, one has the following expression
\begin{align}\label{eq:w_expression}
w_{1,K} = w_{00}+w_{01}(y-y_{j-1})+w_{10}(x-x_{i-1}),
\end{align}
where, $w_{00}$=$w_{1,K}(P_{i-1,j-1})$, $w_{01}$=$w_{1,K}|^{P_{i-1,j}}_{P_{i-1,j-1}}$, and $w_{10}$=$w_{1,K}|^{P_{i,j-1}}_{P_{i-1,j-1}}$.
Substitute this expression into $E_{\mathbb{D},1}$. Notice that $y-\Pi_{h}^{k,*}y \equiv 0$ on edges $\overline{P_{i-1,j-1}P_{i,j-1}}$ and $\overline{P_{i-1,j}P_{i,j}}$, $x-\Pi_{h}^{k,*}x \equiv 0$ on edges $\overline{P_{i-1,j-1}P_{i-1,j}}$ and $\overline{P_{i,j-1}P_{i,j}}$, and $w_{00}\equiv\Pi_{h}^{k,*}w_{00}$ on $K$. Then, with integration by parts, we arrive at
\begin{align*}   
E_{\mathbb{D},1} = &-   \int_{y_{j-1}}^{y_{j}}w_{01}
                        \left( d_{11}\frac{\partial (u-u_{I})}{\partial x} +d_{12}\frac{\partial (u-u_{I})}{\partial y}\right)\Big|_{(x_{i-1},y)}^{(x_{i},y)}
                             (y-\Pi_{h}^{k,*}y)\ud y    \nonumber \\
                   &-   \int_{x_{i-1}}^{x_{i}}w_{10}
                        \left( d_{21}\frac{\partial (u-u_{I})}{\partial x}
                        +d_{22}\frac{\partial (u-u_{I})}{\partial y}\right)\Big|_{(x,y_{j-1})}^{(x,y_{j})}
                            (x-\Pi_{h}^{k,*}x)\ud x     \nonumber \\
                 = & -\int_{K}w_{01}
                        \frac{\partial}{\partial x}\left( d_{11}\frac{\partial (u-u_{I})}{\partial x}
                                                 +d_{12}\frac{\partial (u-u_{I})}{\partial y}\right)
                        (y-\Pi_{h}^{k,*}y)\ud x\ud y  \nonumber\\
                   & -\int_{K}w_{10}
                    \frac{\partial}{\partial y}\left( d_{21}\frac{\partial (u-u_{I})}{\partial x}
                                                 +d_{22}\frac{\partial (u-u_{I})}{\partial y}\right)
                        (x-\Pi_{h}^{k,*}x)\ud x\ud y.
\end{align*}
Combining this with the substitution of (\ref{eq:w_expression}) into $E_{\mathbb{D},2}$, the expansion of differential operators in $E_{\mathbb{D},2}$, and the application of cancellation properties, one obtains
\begin{align*}
E_{\mathbb{D},1}+E_{\mathbb{D},2} =& \int_{K}  w_{01}  \frac{\partial}{\partial y}\Big( d_{21}\frac{\partial (u-u_{I})}{\partial x}
                                              +d_{22}\frac{\partial (u-u_{I})}{\partial y}\Big)
                                                (y-\Pi_{h}^{k,*}y)\ud x \ud y    \nonumber\\
                          &+ \int_{K}  w_{10}  \frac{\partial}{\partial x}\Big( d_{11}\frac{\partial (u-u_{I})}{\partial x}
                                              +d_{12}\frac{\partial (u-u_{I})}{\partial y}\Big)
                                                (x-\Pi_{h}^{k,*}x)\ud x \ud y.
\end{align*}
Recalling (\ref{eq:orth_conditionx})--(\ref{eq:orth_conditiony}), and following similar arguments to those in the proofs of (\ref{eq:E2R2}) and (\ref{eq:E2Q1_2}), we interpolate the coefficients $d_{ij}$ ($i,j=1,2$) and obtain the estimate for
\begin{align}\label{eq:E2D}
|E_{\mathbb{D},1}+E_{\mathbb{D},2}|\le \mathit{C} h^{r+2} \|u\|_{r+3,K} |w_{1,K}|_{1,K}.
\end{align}

Then, (\ref{eq:E2})-(\ref{eq:E2D}) derive the estimate (\ref{eq:Estimation for E2}), which completes the proof.
\end{proof}

\section{Ultraconvergence of derivatives}
\label{sec:Ultraconvergence}
Because the proofs of function-value superconvergence and derivative ultraconvergence are similar, we establish derivative ultraconvergence in this section and defer the superconvergence results for both derivatives and function values to the next section.

We first derive the global ultraconvergence property in the $H^1$ norm in Theorem~\ref{thm:Ultra-super_H1 2D}. Combining this with the superclose estimates between $u$ and $u_{I,Ultra}$ from Lemma~\ref{lem:u_uI_normU}, we obtain the \textbf{natural derivative ultraconvergence} results
\begin{subequations}
\begin{align}
    \label{eq:uh_uI_H1_xUdisc}
 \Big|\frac{\partial }{\partial x} (u - u_{h})\Big|
     = \mathnormal{O}(h^{k+2}), \quad \forall x\in \{\alpha_{m,K}^x\}_{m\in\mathbb{Z}_k},\,\,  \forall   y\in \mathbb{P}_{y,K}^{\mathrm{S}}, \\
    \label{eq:uh_uI_H1_yUdisc}
 \Big|\frac{\partial }{\partial y} (u - u_{h})\Big|
    = \mathnormal{O}(h^{k+2}), \quad \forall x\in \mathbb{P}_{x,K}^{\mathrm{S}},\,\,  \forall  y\in \{\alpha_{m,K}^y\}_{m\in\mathbb{Z}_k},
\end{align}
\end{subequations}
for all $K\in\mathcal{T}_{h}$.
\begin{theorem}[Global ultraconvergence in the $H^{1}$ norm]  \label{thm:Ultra-super_H1 2D}
Given a regular rectangular mesh $\mathcal{T}_{h}$ and a dual mesh $\mathcal{T}_{h}^{*}$ satisfying the tensorial $k$-$k$-order orthogonality condition, suppose $\mathbb{D}=\mathrm{diag}(d_{11},d_{22})$ and $\mathbb{Q}\equiv(0,0)^{T}$ in (\ref{eq:BVP}). Let $u\in H^{1}_{0}(\Omega)\cap H^{k+3}(\Omega)$ be the exact solution of (\ref{eq:BVP}), $u_h\in U_{h}^{k}$ be the solution of the FVE scheme (\ref{eq:FVEscheme}), and $u_{I,Ultra} \in U_{h}^{k}$ be the supercolse function satisfying the AMD-Ultra constraints (Definition~\ref{def:UMD 2D}). Then, there holds the weak estimate of the first type
\begin{align} \label{eq:WeakestimateU}
|a_{h}(u-u_{I,Ultra}, \Pi_{h}^{k,*}w_{h})| \leq \mathit{C}h^{k+2} \|u\|_{k+3} \|w_{h}\|_{\mathcal{T}_h},
                                                          \quad \forall\, w_{h}\in U_{h}^{k},
\end{align}
as well as the \textbf{global ultraconvergence in the $H^1$ norm}
\begin{align} \label{eq:SuperconvU}
\| u_{h}-u_{I,Ultra}\|_{1}\leq \mathit{C}h^{k+2} \|u\|_{k+3}.
\end{align}
\end{theorem}

The global continuity of $u_{I,Ultra} \in U_{h}^{k}$ follows directly from its construction in Definition~\ref{def:UMD 2D}. Then, the ultraconvergence result (\ref{eq:SuperconvU}) follows from the weak estimate (\ref{eq:WeakestimateU}) and Lemma~\ref{lem:Inf-sup condition} through
\begin{align*}
\|u_h-u_{I,Ultra}\|_1 & \leq \frac{1}{c_0} \, \sup_{w_{h}\in U_{h}^{k}}
                         \frac{a_h(u_h-u_{I,Ultra},\Pi_{h}^{k,*}w_{h})} {\|w_{h}\|_{1}}
                       \leq C h^{k+2}\|u\|_{k+3}.
\end{align*}

In the following analysis for (\ref{eq:WeakestimateU}), by the local nature of its estimation, the continuity of $u_{I,Ultra}$ is not required. The terms $R^{\mathrm{U},1}_{x,K}M^y_{1}$, $R^{\mathrm{U},1}_{y,K}M^x_{1}$, and $R^{\mathrm{U}}_{res}$ in (\ref{eq:u-uI_res_U}) represent higher-order residuals  ($O(h^{k+3})$), and their proofs are trivial. Thus, in the derivation of (\ref{eq:WeakestimateU}), we express $u - u_{I,Ultra}$ in the following form, omitting these higher-order terms.
\begin{align}\label{eq:u_uI_simplifyU}
u-u_{I,Ultra} =  R^{\mathrm{S}}_{x,K} +  R^{\mathrm{U}}_{x,K}+R^{\mathrm{S}}_{y,K}+R^{\mathrm{U}}_{y,K}+ R^{\mathrm{S},1}_{x,K} M^y_{1}  +   R^{\mathrm{S},1}_{y,K} M^x_{1}.
\end{align}

\begin{proof}
For any $w_{h}\in U_{h}^{k}$, the left-hand side of (\ref{eq:WeakestimateU}) decomposes into three parts.
\begin{align}\label{eq:Weak_estimate}
a_{h}(u-u_{I,Ultra}, \Pi_{h}^{k,*}w_{h}) = \sum_{K\in\mathcal{T}_{h}} (E_{1}+E_{2}+E_{3}),
\end{align}
where
\begin{align*}
E_{1} &= a^{K}_{h}(u-u_{I,Ultra}, \Pi_{h}^{k,*}(w_{h}-\Pi_{K}^{1}w_{h})),   \\
E_{2} &= a^{K}(u-u_{I,Ultra}, \Pi_{K}^{1}w_{h}),\\
E_{3} &= a^{K}_{h}(u-u_{I,Ultra}, \Pi_{h}^{k,*}(\Pi_{K}^{1}w_{h}))
           -a^{K}(u-u_I, \Pi_{K}^{1}w_{h}).
\end{align*}
Here, $a^{K}_{h}(\cdot,\cdot)$ and $a^{K}(\cdot,\cdot)$ are elementwise bilinear forms of FVEM (\ref{eq:ah_K}) and FEM (\ref{eq:FEM}). The estimate of $E_{3}$ can be referred to Lemma~\ref{lem:Estimation_for_E2} ($r=k$). Following we present the estimate of $E_{1}$ and $E_{2}$.

\textbf{Part I.} For $E_{1}$, by (\ref{eq:ah_K}), it contains three parts that $E_{1}= E_{11}+E_{12}+E_{1\boldsymbol{r}}$ with
\begin{align*}
&E_{11} =  \sum_{ K^{*} \in \mathcal{T}_{h}^{*} }
          -\int_{\partial K^{*}\cap K}d_{11}\frac{\partial (u-u_{I,Ultra})}{\partial x}\Pi_{h}^{k,*}(w_{h}-\Pi_{K}^{1}w_{h})\ud y,\\
&E_{12} =  \sum_{ K^{*} \in \mathcal{T}_{h}^{*} }
          \int_{\partial K^{*}\cap K}d_{22}\frac{\partial (u-u_{I,Ultra})}{\partial y}\Pi_{h}^{k,*}(w_{h}-\Pi_{K}^{1}w_{h})\ud x,\\
&E_{1\boldsymbol{r}} =  \sum_{ K^{*} \in \mathcal{T}_{h}^{*} }
          \int_{K^{*}\cap K}\boldsymbol{r} (u-u_{I,Ultra})\Pi_{h}^{k,*}(w_{h}-\Pi_{K}^{1}w_{h})\ud x \ud y.
\end{align*}
The estimation of $E_{1\boldsymbol{r}}$ can be simply derived by (\ref{eq:w1_norm}) and Lemma~\ref{lem:u_uI_normU} that
\begin{align}\label{eq:E1R_U}
|\sum_{K\in\mathcal{T}_{h}} E_{1\boldsymbol{r}}|\lesssim  h^{k+2}\,||u||_{k+1}\|w_{h}\|_{1}.
\end{align}

The structures of $E_{11}$ and $E_{12}$ are symmetric. Thus, we estimate $E_{11}$ below, with similar results applying to $E_{12}$.
Recall the properties (\ref{eq:Rx_alpha_S}) and (\ref{eq:Rx_alpha_U}), and note that $R^{\mathrm{S}}_{y,K}$ and $R^{\mathrm{U}}_{y,K}$ depend only on $y$. From (\ref{eq:u_uI_simplifyU}), we obtain
\begin{align*}
\frac{\partial (u-u_{I,Ultra})}{\partial x}(\alpha^{x}_{s,K},y)
= \frac{\partial (R^{\mathrm{S},1}_{y,K} M^x_{1})}{\partial x}(\alpha^{x}_{s,K},y)
=\frac{2}{h_{i}^x}  \,  R^{\mathrm{S},1}_{y,K}, \quad \forall s\in\mathbb{Z}_{k}.
\end{align*}
Recalling (\ref{eq:w1_pointy}) and Lemma~\ref{lem:M_function}, since $E_{11}$ consists of line integrals over $x=\alpha^{x}_{s,K}$ ($s\in\mathbb{Z}_{k}$), it can be expressed as
\begin{align}\label{eq:E11_U}
E_{11}=& \sum_{ K^{*} \in \mathcal{T}_{h}^{*} }
   -\int_{\partial K^{*}\cap K}  d_{11}\frac{2}{h_{i}^x}\,R^{\mathrm{S},1}_{y,K}\,\Pi_{h}^{k,*}(w_{h}-\Pi_{K}^{1}w_{h})\ud y      \nonumber   \\
     =& E_{11}^{res,0}
         +  \frac{2\,\bar{d}^{K}_{11}}{h_{i}^x}  \int^{y_{j}}_{y_{j-1}}
                   \,R^{\mathrm{S},1}_{y,K}    \sum_{s=1}^{k} \,
                   \left(\Pi_{h}^{k,*}(w_{h}-\Pi_{K}^{1}w_{h})\right)\Big|_{(a^{x}_{s-1,K},y)}^{(a^{x}_{s,K},y)}\ud y   \nonumber\\
     =&  E_{11}^{res,0}
         +  \frac{2\,\bar{d}^{K}_{11}}{h_{i}^x}  \int^{y_{j}}_{y_{j-1}}
                   \,R^{\mathrm{S},1}_{y,K}     \,
                    \Pi_{h}^{k,*}\left((w_{h}-\Pi_{K}^{1}w_{h})\big|_{(x_{i-1},y)}^{(x_{i},y)}\right)\ud y      \nonumber\\
     =&  E_{11}^{res,0}
         +  \frac{2\,\bar{d}^{K}_{11}}{h_{i}^x}  \int^{y_{j}}_{y_{j-1}}
                   \,R^{\mathrm{S},1}_{y,K}     \,
                    \left(\Pi_{h}^{k,*}w_{h}(x_{i},y) - \Pi_{h}^{k,*}w_{h}(x_{i-1},y) \right)\ud y      \nonumber\\
     =&  E_{11}^{res,0}+E_{11}^{res,r}+E_{11}^{res,l}+E_{11}^{r}+E_{11}^{l}.
\end{align}
where, $\displaystyle\bar{d}^{K}_{11}:=\Pi_{K}^{0}d_{11}$, $\displaystyle\bar{w}_{h}^{x_{i}}:=\frac{1}{h_{j}^{y}}\int_{y_{j-1}}^{y_{j}} w_{h}(x_{i},y) \ud y$, $\displaystyle\bar{d}_{11}^{x_{i}}:=\frac{1}{h_{j}^{y}}\int_{y_{j-1}}^{y_{j}} d_{11}(x_{i},y) \ud y$, and
\begin{align*}
E_{11}^{res,0}=&-\sum_{ K^{*} \in \mathcal{T}_{h}^{*} }
                   \int_{\partial K^{*}\cap K}(d_{11}-\bar{d}^{K}_{11})\frac{2}{h_{i}^x}\,R^{\mathrm{S},1}_{y,K}   \,
                   \Pi_{h}^{k,*}(w_{h}-\Pi_{K}^{1}w_{h})\ud y ,     \nonumber\\
E_{11}^{res,r} =&
           \frac{   2\,(\bar{d}^{K}_{11}-\bar{d}_{11}^{x_{i}}) }{h^{x}_{i}}\int^{y_{j}}_{y_{j-1}}   R^{\mathrm{S},1}_{y,K}  \,
           \Pi_{h}^{k,*}\Big(w_{h}(x_{i},y)-w^{x_{i}}_{0}\Big)\ud y,\\
E_{11}^{res,l} =&
           -\frac{2(\bar{d}^{K}_{11}-\bar{d}_{11}^{x_{i-1}})}{h^{x}_{i}}\int^{y_{j}}_{y_{j-1}}   R^{\mathrm{S},1}_{y,K}\,
           \Pi_{h}^{k,*}\Big(w_{h}(x_{i-1},y)-w^{x_{i-1}}_{0}\Big)\ud y,\\
E_{11}^{r} =&
                2\,\bar{d}_{11}^{x_{i}}\int^{y_{j}}_{y_{j-1}}    \frac{ R^{\mathrm{S},1}_{y,K} }{h^{x}_{i}}\,
           \Pi_{h}^{k,*}\Big(w_{h}(x_{i},y)-w^{x_{i}}_{0}\Big)\ud y,    \\
E_{11}^{l} =&
           -2\,\bar{d}_{11}^{x_{i-1}}  \int^{y_{j}}_{y_{j-1}}  \frac{ R^{\mathrm{S},1}_{y,K}}{h^{x}_{i}} \,
           \Pi_{h}^{k,*}\Big(w_{h}(x_{i-1},y)-w^{x_{i-1}}_{0}\Big)\ud y.
\end{align*}

By the trace theorem, the first three terms $E_{11}^{res,0}$, $E_{11}^{res,r}$, and $E_{11}^{res,l}$ in (\ref{eq:E11_U}) are estimated by
\begin{align}
\label{eq:E11U_1y_0}
\Big| \sum_{K\in\mathcal{T}_{h}}(E_{11}^{res,0}+E_{11}^{res,r}+E_{11}^{res,l}) \Big|\lesssim  h^{k+2}\,||u||_{k+3}\|w_{h}\|_{1}.
\end{align}

As for $E_{11}^{r}$ and $E_{11}^{l}$, recall the expression $R^{\mathrm{S},1}_{y,K} = \sum^{k}_{t=2} b^{*}_{1,t}M^y_{t}  +  b_{1,k+1}M^y_{k+1}$ given by (\ref{eq:u-uI res_S}). Notice the fact that, for any $j\in \mathbb{Z}_{N_{y}}$ and for any $y\in[y_{j-1},\,y_{j}]$,
\begin{align*}
M^{y,K_{i_{1},j}}_{t}(x^{a},y)=M^{y,K_{i_{2},j}}_{t}(x^{b},y),\quad \forall x^{a}\in[x_{i_{1}-1},\,x_{i_{1}}],\,\,\forall x^{b}\in[x_{i_{2}-1},\,x_{i_{2}}].
\end{align*}
Recalling (\ref{eq:M_function_x_Coeff}) in Lemma~\ref{lem:M_function_Coeff} and (\ref{eq:Rx_alpha_S})-(\ref{eq:M-functions_xpointS}), one has
\begin{align*}
\frac{b^{*,K_{i,j}}_{1,t}}{h^{x}_{i}}-\frac{b^{*,K_{i+1,j}}_{1,t}}{h^{x}_{i+1}} = O   \left( \frac{b^{K_{i,j}}_{1,k+1}}{h^{x}_{i}}-\frac{b^{K_{i+1,j}}_{1,k+1}}{h^{x}_{i+1}} \right) = \mathnormal{O}(h^{k+2})
,\quad  t=2,3,\dots,k.
\end{align*}
Combining with $w_{h}|_{\partial \Omega} = 0$ ($w_{h}\in U_{h}^{k}$), we have
\begin{align}\label{eq:elr1U}
\Big| \sum_{K_{i,j}\in\mathcal{T}_{h}} (E_{11}^{r}+E_{11}^{l})  \Big|
= & \Big| \sum_{j=1}^{N_y}  \sum_{i=1}^{N_x-1} 2\,\bar{d}_{11}^{x_{i}}
       \int^{y_{j}}_{y_{j-1}}\Big[
       \sum^{k}_{t=2}\Big(  \frac{b^{*,K_{i,j}}_{1,t}}{h^{x}_{i}}-\frac{b^{*,K_{i+1,j}}_{1,t}}{h^{x}_{i+1}}  \Big)M^y_{t}   \nonumber\\
  & \,\,\,     +\Big(  \frac{b^{K_{i,j}}_{1,k+1}}{h^{x}_{i}}-\frac{b^{K_{i+1,j}}_{1,k+1}}{h^{x}_{i+1}}  \Big) M^y_{k+1} \Big] \,
        \Pi_{h}^{k,*}\big(w_{h}(x_{i},y)-w^{x_{i}}_{0}\big)\ud y    \Big|   \nonumber\\
\lesssim&  h^{k+2}\,||u||_{k+3}\|w_{h}\|_{1},
\end{align}
which together with (\ref{eq:E11_U})-(\ref{eq:E11U_1y_0}) derive
\begin{align}\label{eq:E11U_estimate}
|E_{11}|\lesssim h^{k+2}\,||u||_{k+3}\|w_{h}\|_{1}.
\end{align}

Similarly,
\begin{align}\label{eq:E12U_estimate}
|E_{12}|\lesssim h^{k+2}\,||u||_{k+3}\|w_{h}\|_{1}.
\end{align}

Then, one arrives at the estimate of $E_{1}$.
\begin{align} \label{eq:E1_U_estimation}
\Big|\sum_{K\in\mathcal{T}_h} E_{1} \Big| = \Big| \sum_{K\in\mathcal{T}_h}(E_{11}+E_{12}+E_{1\boldsymbol{r}} ) \Big|
\lesssim \mathit{C}h^{k+2} \|u\|_{k+3} \|w_{h}\|_{1}.
\end{align}

\textbf{Part II.} For $E_2$, it also contains three parts that $E_{2}= E_{21}+E_{22}+E_{2\boldsymbol{r}}$ with
\begin{align*}
&E_{21} =
                           \int_{K}d_{11}\frac{\partial (u-u_{I,Ultra})}{\partial x}
                           \frac{\partial \, \Pi_{K}^{1}w_{h}}{\partial x}\ud x \ud y,\\
&E_{22} =
                           \int_{K}d_{22}\frac{\partial (u-u_{I,Ultra})}{\partial y}
                           \frac{\partial \, \Pi_{K}^{1}w_{h} }{\partial y}\ud x \ud y,\\
&E_{2\boldsymbol{r}}             = \int_{K}\boldsymbol{r}(u-u_{I,Ultra})\Pi_{K}^{1}w_{h}\ud x \ud y.
\end{align*}

For $E_{21}$, notice that $\displaystyle \frac{\partial\, \Pi_{K}^{1}w_{h}}{\partial x}= \frac{w_{h}(x_{i},y_{j-1})-w_{h}(x_{i-1},y_{j-1})}{h_{i}^{x}}$ is a constant on $K$.
Recall the properties (\ref{eq:M-functions_xpointS}), (\ref{eq:M-functions_xpointU}), (\ref{eq:u_uI_simplifyU}), and note that $R^{\mathrm{S}}_{y,K}$ and $R^{\mathrm{U}}_{y,K}$ depend only on $y$.
By Lemma~\ref{lem:M_function} and the integration by parts in $x$-direction, one obtains
\begin{align}\label{eq:E3D11_U_estimation}
\Big| \sum_{K\in\mathcal{T}_{h}}  E_{21} \Big|
                                   =& \Big| \sum_{K\in\mathcal{T}_{h}} \frac{\partial\, \Pi_{K}^{1}w_{h}}{\partial x}   \Big[
                                       \int_{K}d_{11}
                                       \frac{\partial (R^{\mathrm{S}}_{x,K} +  R^{\mathrm{U}}_{x,K}+ R^{\mathrm{S},1}_{x,K} M^y_{1})}{\partial x}  \ud x \ud y  \nonumber\\
                                     &  \qquad\qquad       +    \int_{K}d_{11}\frac{\partial (R^{\mathrm{S},1}_{y,K} M^x_{1}) }{\partial x}     \ud x \ud y \Big]    \Big| \nonumber  \\
                                   =& \Big| \sum_{K\in\mathcal{T}_{h}} \frac{\partial\, \Pi_{K}^{1}w_{h}}{\partial x}   \Big[
                                       \int_{K}(R^{\mathrm{S}}_{x,K} +  R^{\mathrm{U}}_{x,K}+ R^{\mathrm{S},1}_{x,K} M^y_{1})
                                       \frac{\partial d_{11}}{\partial x}  \ud x \ud y  \nonumber\\
                                    &  \qquad\qquad        +    \int_{K}d_{11}\frac{\partial  M^x_{1} }{\partial x}  R^{\mathrm{S},1}_{y,K}   \ud x \ud y \Big]    \Big| \nonumber  \\
                                  =& \Big| \sum_{K\in\mathcal{T}_{h}} \frac{\partial\, \Pi_{K}^{1}w_{h}}{\partial x}   \Big[
                                       \int_{K}R^{\mathrm{S}}_{x,K}
                                       \Big( \frac{\partial d_{11}}{\partial x} - \Pi_{K}^{0}\frac{\partial d_{11}}{\partial x}\Big)  \ud x \ud y  \nonumber\\
                                    &   +\int_{K}(R^{\mathrm{U}}_{x,K}+ R^{\mathrm{S},1}_{x,K} M^y_{1}) \frac{\partial d_{11}}{\partial x}\ud x \ud y
                                    +    \int_{K}(d_{11}-\bar{d}^{K}_{11}) \frac{\partial  M^x_{1} }{\partial x}  R^{\mathrm{S},1}_{y,K}   \ud x \ud y \Big]    \Big| \nonumber  \\
                            \lesssim & h^{k+2}\,||u||_{k+3}\|w_{h}\|_{1}.
\end{align}

Similarly, one has the estimation of $E_{22}$.
\begin{align}\label{eq:E3D22_U_estimation}
\Big|  \sum_{K\in\mathcal{T}_{h}} E_{22}  \Big| \lesssim h^{k+2}\,||u||_{k+3}\|w_{h}\|_{1}.
\end{align}

For $E_{2\boldsymbol{r}}$, by Lemma~\ref{lem:M_function} and (\ref{eq:u_uI_simplifyU}), denoting $\bar{\boldsymbol{r}}^{K}:=\Pi_{K}^{0}\boldsymbol{r}$, we arrive at
\begin{align}\label{eq:E3R_U_estimation}
\Big| \sum_{K\in\mathcal{T}_{h}} E_{2\boldsymbol{r}} \Big|
        =&  \Big| \sum_{K\in\mathcal{T}_{h}} \Big( \int_{K}\boldsymbol{r} (  R^{\mathrm{U}}_{x,K}+R^{\mathrm{U}}_{y,K}+ R^{\mathrm{S},1}_{x,K} M^y_{1}  +   R^{\mathrm{S},1}_{y,K} M^x_{1})\Pi_{K}^{1}w_{h}\ud x \ud y  \nonumber\\
         &  \qquad  + \int_{K} (\boldsymbol{r} - \bar{\boldsymbol{r}}^{K} )  (R^{\mathrm{S}}_{x,K} +R^{\mathrm{S}}_{y,K})\Pi_{K}^{1}w_{h}\ud x \ud y      \nonumber\\
         &  \qquad  + \int_{K} \bar{\boldsymbol{r}}^{K}  (R^{\mathrm{S}}_{x,K} +R^{\mathrm{S}}_{y,K}) (\Pi_{K}^{1}w_{h}-\Pi_{K}^{0}\Pi_{K}^{1}w_{h})\ud x \ud y \Big)  \Big|    \nonumber\\
 \lesssim  &  h^{k+2}\,||u||_{k+3}\|w_{h}\|_{1}.
\end{align}

Thus, we have the estimate of $E_{2}$ that
\begin{align} \label{eq:E3_U_estimation}
\Big|\sum_{K\in\mathcal{T}_h} E_{2} \Big| = \Big| \sum_{K\in\mathcal{T}_h}(E_{21}+E_{22}+E_{2\boldsymbol{r}} ) \Big|
\lesssim \mathit{C}h^{k+2} \|u\|_{k+3} \|w_{h}\|_{1},
\end{align}
which together with (\ref{eq:Weak_estimate}), (\ref{eq:E1_U_estimation}) and Lemma~\ref{lem:Estimation_for_E2} complete the proof.
\end{proof}

\section{Superconvergence of derivatives and function values}
\label{sec:Superconvergence}
This section establishes the standard one-order-higher superconvergence properties in the $H^1$ norm (Theorem~\ref{thm:super_H1_2D}) and the $L^2$ norm (Theorem~\ref{thm:super_L2_2D}). Together with the superclose estimates between $u$ and $u_{I,Super}$ from Lemma~\ref{lem:u_uI_normS}, these yield the \textbf{natural derivative superconvergence} results
\begin{subequations}
\begin{align}
    \label{eq:uh_uI_H1_xSdisc}
 \Big|\frac{\partial }{\partial x} (u - u_{h})\Big|
     = \mathnormal{O}(h^{k+1}), \quad \forall x\in \{\alpha_{m,K}^x\}_{m\in\mathbb{Z}_k},        \,\,  \forall y\in [y_{j-1},y_{j}], \\
    \label{eq:uh_uI_H1_ySdisc}
 \Big|\frac{\partial }{\partial y} (u - u_{h})\Big|
    = \mathnormal{O}(h^{k+1}), \quad \forall x\in [x_{i-1}, x_{i}],     \,\,\forall y\in \{\alpha_{m,K}^y\}_{m\in\mathbb{Z}_k},
\end{align}
\end{subequations}
and \textbf{natural function-value superconvergence} result
\begin{align}\label{eq:uh_uI_L2_Sdisc}
      |u - u_{h}| = \mathnormal{O}(h^{k+2}),\quad \forall x\in\mathbb{P}_{x,K}^{\mathrm{S}},\,\, \forall y\in \mathbb{P}_{y,K}^{\mathrm{S}},
\end{align}
for all $K\in\mathcal{T}_{h}$.

\begin{theorem}[Global superconvergence in $H^{1}$ norm]  \label{thm:super_H1_2D}
Given a regular rectangular mesh $\mathcal{T}_{h}$ and a dual mesh $\mathcal{T}_{h}^{*}$ satisfying the tensorial $k$-$(k-1)$-order orthogonality condition. Let $u\in H^{1}_{0}(\Omega)\cap H^{k+2}(\Omega)$ be the exact solution of (\ref{eq:BVP}), $u_h\in U_{h}^{k}$ be the solution of the FVE scheme (\ref{eq:FVEscheme}), and $u_{I,Super} \in U_{h}^{k}$ be the supercolse function satisfying the AMD-Super constraints (Definition~\ref{def:SMD 2D}). Then, there holds the weak estimate of the first type
\begin{equation} \label{eq:WeakestimateS}
|a_{h}(u-u_{I,Super}, \Pi_{h}^{k,*}w_{h})| \leq \mathit{C}h^{k+1} \|u\|_{k+2} \|w_{h}\|_{1},
                                                            \quad \forall\, w_{h}\in U_{h}^{k},
\end{equation}
as well as the \textbf{global superconvergence in the $H^1$ norm}
\begin{equation} \label{eq:SuperconvS}
\| u_{h}-u_{I,Super}\|_{1} \leq \mathit{C}h^{k+1} \|u\|_{k+2}.
\end{equation}
\end{theorem}

The global superconvergence in the $H^1$ norm (\ref{eq:SuperconvS}) can be derived by (\ref{eq:WeakestimateS}) and Lemma~\ref{lem:Inf-sup condition}.
In the following analysis for (\ref{eq:WeakestimateS}), by the local nature of its estimation, the continuity of $u_{I,Super}$ is not required. The terms $R^{\mathrm{S},1}_{x,K}M^y_{1}$, $R^{\mathrm{S},1}_{y,K}M^x_{1}$, and $R^{\mathrm{S}}_{res}$ in (\ref{eq:u-uI res_S}) represent higher-order residuals ($O(h^{k+2})$), and their proofs are trivial. Thus, in the derivation of (\ref{eq:WeakestimateS}), we express $u - u_{I,Super}$ in the following form, omitting these higher-order terms.
\begin{align}\label{eq:u_uI_simplifyS}
u-u_{I,Super} =  R^{\mathrm{S}}_{x,K}   +  R^{\mathrm{S}}_{y,K}.
\end{align}

\begin{proof}
Similar with (\ref{eq:Weak_estimate}), the left-hand side of (\ref{eq:WeakestimateS}) decomposes into three parts.
\begin{align}\label{eq:Weak estimate_splitS}
a_{h}(u-u_{I,Super}, \Pi_{h}^{k,*}w_{h}) = \sum_{K\in\mathcal{T}_{h}} (E^{S}_1+E^{S}_2+E^{S}_3),
\end{align}
where
\begin{align*}
E^{\mathrm{S}}_1 &= a^{K}_{h}(u-u_{I,Super}, \Pi_{h}^{k,*}(w_{h}-\Pi_{K}^{1}w_{h})),\\
E^{\mathrm{S}}_2 &= a^{K}(u-u_{I,Super}, \Pi_{K}^{1}w_{h}),  \\
E^{\mathrm{S}}_3 &= a^{K}_{h}(u-u_{I,Super}, \Pi_{h}^{k,*} \Pi_{K}^{1}w_{h})-a^{K}(u-u_{I,Super}, \Pi_{K}^{1}w_{h}).
\end{align*}

The estimate of $E^{\mathrm{S}}_{3}$ can be referred to Lemma~\ref{lem:Estimation_for_E2} for the case of $r=k-1$.

\textbf{Part I.} For $E^{\mathrm{S}}_1$, by (\ref{eq:u_uI_simplifyS}), one has
\begin{align*}
\frac{\partial (u-u_{I,Ultra})}{\partial x}
= \frac{\partial R^{\mathrm{S}}_{x,K} }{\partial x},  \qquad
\frac{\partial (u-u_{I,Ultra})}{\partial y}
= \frac{\partial R^{\mathrm{S}}_{y,K} }{\partial y}.
\end{align*}
Then, recalling the properties (\ref{eq:w1_pointy}), (\ref{eq:Rx_alpha_S}), and Lemma~\ref{lem:M_function} we have
\begin{align}\label{eq:E1S}
E^{\mathrm{S}}_1    =&
                      \sum_{ K^{*} \in \mathcal{T}_{h}^{*} }
          -\int_{\partial K^{*}\cap K} \Big(  \bar{d}^{K}_{11}\frac{\partial R^{\mathrm{S}}_{x,K} }{\partial x} + \bar{d}^{K}_{12}\frac{\partial R^{\mathrm{S}}_{y,K} }{\partial y} \Big) \Pi_{h}^{k,*}(w_{h}-\Pi_{K}^{1}w_{h})\ud y        \nonumber \\
                    &+  \sum_{ K^{*} \in \mathcal{T}_{h}^{*} }
          \int_{\partial K^{*}\cap K} \Big(  \bar{d}^{K}_{21}\frac{\partial R^{\mathrm{S}}_{x,K} }{\partial x} + \bar{d}^{K}_{22}\frac{\partial R^{\mathrm{S}}_{y,K} }{\partial y} \Big)\Pi_{h}^{k,*}(w_{h}-\Pi_{K}^{1}w_{h})\ud x     \nonumber\\
                &+E^{\mathrm{S}}_{1,res} + E^{\mathrm{S}}_{1\boldsymbol{r}}      \nonumber\\
        =&  \int^{y_{j}}_{y_{j-1}}\,\bar{d}^{K}_{12}
                   \,\frac{\partial R^{\mathrm{S}}_{y,K} }{\partial y}    \sum_{s=1}^{k} \,
                        \left(\Pi_{h}^{k,*}(w_{h}-\Pi_{K}^{1}w_{h})\right)\Big|_{(a^{x}_{s-1,K},y)}^{(a^{x}_{s,K},y)}\ud y   \nonumber\\
            &+  \int^{x_{i}}_{x_{i-1}}\,\bar{d}^{K}_{21}
                   \,\frac{\partial R^{\mathrm{S}}_{x,K} }{\partial x}    \sum_{t=1}^{k} \,
                     \left(\Pi_{h}^{k,*}(w_{h}-\Pi_{K}^{1}w_{h})\right)\Big|_{(x,a^{y}_{t-1,K})}^{(x,a^{y}_{t,K})}\ud x
                +E^{\mathrm{S}}_{1,res} + E^{\mathrm{S}}_{1\boldsymbol{r}}      \nonumber\\
        =& E^{\mathrm{S}}_{11} + E^{\mathrm{S}}_{12}+  E^{\mathrm{S}}_{1,res} +E^{\mathrm{S}}_{1\boldsymbol{r}},
\end{align}
where
\begin{align*}
E^{\mathrm{S}}_{11}        =&  \bar{d}^{K}_{12} \int^{y_{j}}_{y_{j-1}}
                   \,\frac{\partial R^{\mathrm{S}}_{y,K} }{\partial y}     \,
                   \left(\Pi_{h}^{k,*}w_{h}\right)\Big|_{({x}_{i-1},y)}^{({x}_{i},y)}\ud y,       \\
E^{\mathrm{S}}_{12}        =&\bar{d}^{K}_{21}  \int^{x_{i}}_{x_{i-1}}
                   \,\frac{\partial R^{\mathrm{S}}_{x,K} }{\partial x}     \,
                   \left(\Pi_{h}^{k,*}w_{h}\right)\Big|_{(x,y_{j-1})}^{(x,y_{j})}\ud x,       \\
E^{\mathrm{S}}_{1,res}  =&  \sum_{ K^{*} \in \mathcal{T}_{h}^{*} }
          -\int_{\partial K^{*}\cap K} \Big( (\mathbb{D}-\Pi_{K}^{0}\mathbb{D}) \nabla(u-u_{I,Super})  \Big)\cdot \boldsymbol{n} \, \Pi_{h}^{k,*}(w_{h}-\Pi_{K}^{1}w_{h})\ud s,                    \\
E^{\mathrm{S}}_{1\boldsymbol{r}}=&  \sum_{ K^{*} \in \mathcal{T}_{h}^{*} }
          \int_{K^{*}\cap K}\Big( \mathbb{Q}\cdot\nabla (u-u_{I,Super})+ \boldsymbol{r} (u-u_{I,Super})\Big) \nonumber\\
            &\qquad\qquad \Pi_{h}^{k,*}(w_{h}-\Pi_{K}^{1}w_{h})\ud x \ud y.
\end{align*}

For $E^{\mathrm{S}}_{11}$, as $u-u_{I,Super}$ (\ref{eq:u-uI res_S}) is continuous across $\overline{P_{i,j-1}P_{i,j}}$, it takes the same value on this segment, the right boundary of $K_{i,j}$, and the left boundary of $K_{i+1,j}$, for all $j\in\mathbb{Z}_{N_{y}}$ and $i\in\mathbb{Z}_{N_{x}-1}$.
\begin{align*}
(u - u_{I,Super})|_{(x_{i},y)}= R^{\mathrm{S}}_{y,K_{i,j}} + R^{\mathrm{S},1}_{y,K_{i,j}} = R^{\mathrm{S}}_{y,K_{i+1,j}} - R^{\mathrm{S},1}_{y,K_{i+1,j}}.
\end{align*}
This property leads to a cancellation of the first term on the righthand side of the following equation that
\begin{align*}
\Big| \sum_{K\in\mathcal{T}_{h}} E^{\mathrm{S}}_{11} \Big|
         =& \Big| \sum_{K_{i,j}\in\mathcal{T}_{h}}   \int^{y_{j}}_{y_{j-1}}  \sum_{s=0,1}\Big[ \Big( (-1)^{s}\bar{d}^{x_{i-s}}_{12}
                   \,\frac{\partial (u - u_{I,Super}) }{\partial y}     \,
                        \Pi_{h}^{k,*}w_{h}  \Big)\Big|_{({x}_{i-s},y)}    \\
          & \qquad +  \Big(-\bar{d}^{x_{i-s}}_{12}
                   \,\frac{\partial R^{\mathrm{S},1}_{y,K} }{\partial y}     \,
                        \Pi_{h}^{k,*}w_{h}  \Big) \Big|_{({x}_{i-s},y)}     \\
          & \qquad +  \Big( (-1)^{s}(\bar{d}^{K}_{12}-\bar{d}^{x_{i-s}}_{12})
                   \,\frac{\partial R^{\mathrm{S}}_{y,K} }{\partial y}     \,
                        \Pi_{h}^{k,*}w_{h}  \Big)  \Big|_{({x}_{i-s},y)}  \Big] \ud y   \Big|  \\
         =& \Big| \sum_{K\in\mathcal{T}_{h}}    \int^{y_{j}}_{y_{j-1}}  \sum_{s=0,1} \Big[\Big(- \bar{d}^{x_{i-s}}_{12}
                   \,\frac{\partial R^{\mathrm{S},1}_{y,K} }{\partial y}     \,
                        \Pi_{h}^{k,*}(w_{h}-\bar{w}_{h}^{K})  \Big) \Big|_{({x}_{i-s},y)}   \\
          &  \qquad +       \Big( (-1)^{s}(\bar{d}^{K}_{12}-\bar{d}^{x_{i-s}}_{12})
                   \,\frac{\partial R^{\mathrm{S}}_{y,K} }{\partial y}     \,
                        \Pi_{h}^{k,*}(w_{h}-\bar{w}_{h}^{K})  \Big)  \Big|_{({x}_{i-s},y)}  \Big]  \ud y   \Big|  \\
          \lesssim & h^{k+1} \|u\|_{k+2} \|w_{h}\|_{1}.
\end{align*}

The estimate of $E^{\mathrm{S}}_{12}$ is quite similar. While, estimates for $E^{\mathrm{S}}_{1,res}$ and $E^{\mathrm{S}}_{1\boldsymbol{r}}$ are given by
\begin{align*}
\Big| \sum_{K\in\mathcal{T}_{h}}(E^{\mathrm{S}}_{1,res}+ E^{\mathrm{S}}_{1\boldsymbol{r}}) \Big| \lesssim h^{k+1} \|u\|_{k+2} \|w_{h}\|_{1}.
\end{align*}
which derives the estimate of $E^{\mathrm{S}}_{1}$ that
\begin{align} \label{eq:E1S_estimation}
\Big| \sum_{K\in\mathcal{T}_{h}} E^{\mathrm{S}}_{1} \Big|
         \lesssim \Big| \sum_{K\in\mathcal{T}_{h}} (E^{\mathrm{S}}_{11} + E^{\mathrm{S}}_{12}+  E^{\mathrm{S}}_{1,res} +E^{\mathrm{S}}_{1\boldsymbol{r}}) \Big|
         \lesssim  h^{k+1} \|u\|_{k+2} \|w_{h}\|_{1}.
\end{align}

\textbf{Part II.} For $E^{\mathrm{S}}_{2}$, by (\ref{eq:M-functions_xpointS}), (\ref{eq:u_uI_simplifyS}) and the integration by parts in $x$-direction or $y$-direction, one has
\begin{align}\label{eq:E3S}
\Big| \sum_{K\in\mathcal{T}_{h}} E^{\mathrm{S}}_{2} \Big|
                        =& \Big| \sum_{K\in\mathcal{T}_{h}}     \Big[
                          \int_{K}\frac{\partial R^{\mathrm{S}}_{x,K}}{\partial x}  \Big(d_{11}\frac{\partial \Pi_{K}^{1}w_{h} }{\partial x} +  d_{21}\frac{\partial \Pi_{K}^{1}w_{h} }{\partial y} +q_{1}\, \Pi_{K}^{1}w_{h} \Big)  \ud x \ud y    \nonumber \\
                         &  \qquad
                           + \int_{K}\frac{\partial R^{\mathrm{S}}_{y,K}}{\partial y}  \Big(d_{12}\frac{\partial \Pi_{K}^{1}w_{h} }{\partial x} +  d_{22}\frac{\partial \Pi_{K}^{1}w_{h} }{\partial y} +q_{2}\, \Pi_{K}^{1}w_{h} \Big)  \ud x \ud y    \nonumber \\
                         & \qquad  + \int_{K}\boldsymbol{r} (u-u_{I,Super})\, \Pi^{1}_{K}w_{h}  \ud x \ud y    \Big] \Big|          \nonumber \\
                        =& \Big| \sum_{K\in\mathcal{T}_{h}}     \Big[
                          \int_{K}R^{\mathrm{S}}_{x,K}  \frac{\partial }{\partial x} \Big(d_{11}\frac{\partial \Pi_{K}^{1}w_{h} }{\partial x} +  d_{21}\frac{\partial \Pi_{K}^{1}w_{h} }{\partial y} +q_{1}\, \Pi_{K}^{1}w_{h} \Big) \ud x \ud y    \nonumber \\
                         &  \qquad
                           + \int_{K}R^{\mathrm{S}}_{y,K} \frac{\partial }{\partial y}  \Big(d_{12}\frac{\partial \Pi_{K}^{1}w_{h} }{\partial x} +  d_{22}\frac{\partial \Pi_{K}^{1}w_{h} }{\partial y} +q_{2}\, \Pi_{K}^{1}w_{h} \Big)  \ud x \ud y    \nonumber \\
                         & \qquad  + \int_{K}\boldsymbol{r} (u-u_{I,Super})\, \Pi^{1}_{K}w_{h}  \ud x \ud y    \Big] \Big|  \nonumber \\
                \lesssim & h^{k+1} \|u\|_{k+2} \|w_{h}\|_{1} .
\end{align}

Then, we arrive at the proof of (\ref{eq:WeakestimateS}).
\end{proof}

\begin{theorem}[Global superconvergence in the $L^{2}$ norm]  \label{thm:super_L2_2D}
Given a regular rectangular mesh $\mathcal{T}_{h}$ and a dual mesh $\mathcal{T}_{h}^{*}$ satisfying the tensorial $k$-$k$-order orthogonality condition. Let $u\in H^{1}_{0}(\Omega)\cap H^{k+3}(\Omega)$ be the exact solution of (\ref{eq:BVP}), $u_h\in U_{h}^{k}$ be the solution of the FVE scheme (\ref{eq:FVEscheme}), and $u_{I,Super} \in U_{h}^{k}$ be the supercolse function satisfying the AMD-Super constraints (Definition~\ref{def:SMD 2D}). Then, there holds the \textbf{global superconvergence in the $L^2$ norm}
\begin{align}\label{eq:SuperconvL2_S}
\| u_{h}-u_{I,Super}\|_{0}\leq \mathit{C}h^{k+2} \|u\|_{k+3}.
\end{align}
\end{theorem}

\begin{proof}
Consider the auxiliary problem: for any $g\in L^{2}(\Omega)$, find $w \in H^{1}_{0}(\Omega)$ such that
\begin{align} \label{eq:dualproblem}
a(v,w) = (g,v),\quad \forall v\in H^{1}_{0}(\Omega).
\end{align}
The regularity of this auxiliary problem shows that $\|w\|_{2} \leq C \|g\|_{0}$ for any $w\in H_{0}^{1}(\Omega)\cap H^2(\Omega)$.

Taking $v = g = u_h-u_{I,Super}$ in (\ref{eq:dualproblem}), by the orthogonality of the FVE scheme $a_{h}(u-u_{h}, v_{h}) = 0$ for any $v_{h}\in V_{h}$, we have
\begin{align}  \label{eq:uh_hI_L2}
\|u_h - u_{I,Super}\|_{0}^2 =& \sum_{K\in\mathcal{T}_{h}}   a^{K}(u_h - u_{I,Super}, w-\Pi_{K}^{1}w)   \nonumber \\
        &  +  \sum_{K\in\mathcal{T}_{h}}   a^{K}(u - u_{I,Super}, \Pi_{K}^{1}w)     \nonumber \\
        &  +  \sum_{K\in\mathcal{T}_{h}}   \Big( a_{h}^{K}(u-u_h, \Pi_h^{k,*} \Pi_{K}^{1}w)-a^{K}(u - u_h, \Pi_{K}^{1}w) \Big).
\end{align}

The estimate of the first part of (\ref{eq:uh_hI_L2}) follows from (\ref{eq:SuperconvS}) that
\begin{align*}
|\sum_{K\in\mathcal{T}_{h}}   a^{K}(u_h - u_{I,Super}, w-\Pi_{K}^{1}w)| \lesssim  h^{k+2} \|u\|_{k+2}\|w\|_{2}.
\end{align*}
The second part of (\ref{eq:uh_hI_L2}) can be estimated by an argument similar to Part II of Theorem~\ref{thm:Ultra-super_H1 2D}, and the third part of (\ref{eq:uh_hI_L2}) is estimated by Lemma~\ref{lem:Estimation_for_E2} ($r=k$). This completes the proof of Theorem~\ref{thm:super_L2_2D}.
\end{proof}

\section{Numerical experiments}
\label{sec:Numerical experiments}
In this section, we consider three coefficient patterns for (\ref{eq:BVP}), listed in Table~\ref{tab:equations}:

$\bullet$ BVP-D: pure unit diffusion;

$\bullet$ BVP-DR: diagonal diffusion-reaction with variable coefficients

(used to test derivative ultraconvergence);

$\bullet$ BVP-DQR: full-tensor convection-diffusion-reaction with variable coefficients

(used to test derivative and function-value superconvergence).

The source term $f$ is chosen such that $u(x,y) = \sin(\pi x)\sin(2\pi y)e^{x - 0.5 + y^2}$.
Table~\ref{tab:FVEschemes} gives the FVE-$k$-$r$ schemes satisfying the tensorial $k$-$r$-order orthogonality condition and are employed in the computations.

\begin{table}[htbp!]
  \caption{The BVP equations}\label{tab:equations}
  \resizebox{\textwidth}{!}
  {
  \begin{tabular}{ l  l l  l  }
    1. BVP-D&     $\mathbb{D} = I_{2\times 2}$    &$\mathbb{Q} = 0$   &  $\boldsymbol{r} = 0$    \\
    2. BVP-DR&    $\mathbb{D} = \mathrm{diag}( ye^{x}+1,\,  xe^{y}+1)$
                                & $\mathbb{Q} = 0$       & $\boldsymbol{r} = (x+1)(y+1)$    \\
    3. BVP-DQR&   $\mathbb{D} = \left(
                 \begin{array}{cc}
                   ye^{x}+1   & xy \\
                   xy         & xe^{y}+1 \\
                 \end{array}
               \right)$
                        & $\mathbb{Q} = (\cos x,\cos y)^{T}$
                        & $\boldsymbol{r} = (x+1)(y+1)$
  \end{tabular}
  }
\end{table}
\begin{table}[htbp!]
\renewcommand\arraystretch{1.5}
\centering
  \caption{The FVE-$k$-$r$ schemes}\label{tab:FVEschemes}
  \resizebox{\textwidth}{!}
  {
  \begin{tabular}{ c| c | c | c |c  }
  \toprule
  $k$ & $r$ & {\em Schemes} & Dual strategies  & Interpolation nodes \\
  \hline
  \multirow{3}{*}[-2.0ex]{3}  &  2  & FVE-3-2  & \makecell[c]{$\boldsymbol{\alpha}^{x}\approx(-0.6406,-0.0748,0.6255)$ \\
                                                             $\boldsymbol{\alpha}^{y}\approx(-0.7622,-0.2073,0.6577)$}
                                              & \makecell[c]{ $(-1/5,7/50)$ \\
                                                              $(-1/2,1/5)$}\\
  \cline{2-5}
                      &  3  & FVE-3-3  & \makecell[c]{$\boldsymbol{\alpha}^{x}\approx(-0.8563,-0.1534,0.7243)$ \\
                                                             $\boldsymbol{\alpha}^{y}\approx(-0.9380,-0.2435,0.7011)$}
                                              & \makecell[c]{ $(-3/5,\approx 0.3301)$ \\
                                                              $(-5/7,\approx 0.2744)$}\\
  \cline{2-5}
                      &  4  & \makecell[c]{FVE-3-4 }
                            &  $\boldsymbol{\alpha}^{x}=\boldsymbol{\alpha}^{y}=(-\sqrt{3/5},0,\sqrt{3/5})$
                                                & Gaussian-Duality\\
  \hline
  \multirow{3}{*}[-2.0ex]{4}  &  3  & FVE-4-3  & \makecell[c]{$\boldsymbol{\alpha}^{x}\approx(-0.9156,-0.2698,0.5678,0.8838)$ \\
                                                             $\boldsymbol{\alpha}^{y}\approx(-0.9020,-0.3187,0.4628,0.8990)$}
                                              & \makecell[c]{ $(-7/10,1/5,4/5)$ \\
                                                              $(-7/10,1/10,3/4)$}\\
  \cline{2-5}
                      &  4  & FVE-4-4  & \makecell[c]{$\boldsymbol{\alpha}^{x}\approx(-0.9579,-0.4479,0.3699,0.9093)$ \\
                                                             $\boldsymbol{\alpha}^{y}\approx(-0.8598,-0.2885,0.4744,0.9452)$}
                                              & \makecell[c]{ $(-4/5,-1/25,\approx 0.7270)$ \\
                                                              $(-16/25,1/10,\approx 0.7960)$}\\
  \cline{2-5}
                      &  6  & \makecell[c]{FVE-4-6 }
                            &  $\boldsymbol{\alpha}^{x}=\boldsymbol{\alpha}^{y}=(-\sqrt{  \frac{15+2\sqrt{30}}{35} },  -\sqrt{  \frac{15-2\sqrt{30}}{35} },\sqrt{  \frac{15-2\sqrt{30}}{35} },\sqrt{  \frac{15+2\sqrt{30}}{35} })$
                                       & Gaussian-Duality\\
  \bottomrule
  \end{tabular}
  }
  \begin{tablenotes}
      \footnotesize
      \item[1] The decimals after $"\approx"$ are the divisors which retain four significant digits.
  \end{tablenotes}
\end{table}

We test the following discrete norms for $e_h := u - u_h$:
\begin{align*}
\|e_h\|_{H^{1},x}^{\mathrm{S}}
&= \left( \sum_{K \in \mathcal{T}_{h}} h_i^x \int_{y_{j-1}}^{y_j} \sum_{s=1}^{k}
    \left(\frac{\partial e_h}{\partial x}(\alpha^{x}_{s,K}, y)\right)^2 dy \right)^{1/2}, \\
\|e_h\|_{L^{2}}^{\mathrm{S}}
&= \left( \sum_{K \in \mathcal{T}_{h}} \frac{|K|}{(k+1)^2}
    \sum_{x \in \mathbb{P}_{x,K}^{\mathrm{S}}} \sum_{y \in \mathbb{P}_{y,K}^{\mathrm{S}}} e_h(x,y)^2 \right)^{1/2}, \\
\|e_h\|_{H^{1},x}^{\mathrm{U}}
&= \left( \sum_{K \in \mathcal{T}_{h}} \frac{|K|}{k(k+1)}
    \sum_{s \in \mathbb{Z}_{k}} \sum_{y \in \mathbb{P}_{y,K}^{\mathrm{S}}}
    \left(\frac{\partial e_h}{\partial x}(\alpha^{x}_{s,K},y)\right)^2 \right)^{1/2},
\end{align*}
which quantify derivative superconvergence, function-value superconvergence, and derivative ultraconvergence, respectively. Here, we only present results for the $x$-derivative, as the $y$-derivative results are identical by symmetry.

\begin{table}[htbp!]
\centering
  \caption{Derivative ultraconvergence error of Example~\ref{exam:UltraH1_exam}}\label{tab:Ultra_k}
  \resizebox{\textwidth}{!}
  {
  \begin{tabular}{ c | c c | c c |c c | c c }
  \toprule
\multicolumn{1}{c|}{BVP}    & \multicolumn{2}{c|}{ BVP-DR } & \multicolumn{6}{c}{ BVP-D } \\
  \toprule
Scheme       & \multicolumn{2}{c|}{ FVE-3-3 } & \multicolumn{2}{c|}{ FVE-3-2 }
       & \multicolumn{2}{c|}{ FVE3-4 } & \multicolumn{2}{c}{ FE-3 } \\
  \hline
  $h$  & $ \|u-u_{h}\|_{H^{1},x}^{\mathrm{U}} $ & Order & $ \|u-u_{h}\|_{H^{1},x}^{\mathrm{U}} $ & Order
       & $ \|u-u_{h}\|_{H^{1},x}^{\mathrm{U}} $ & Order & $ \|u-u_{h}\|_{H^{1},x}^{\mathrm{U}} $ & Order\\
  \hline
  1/12 & 2.9363E-06 & $\setminus$ & 2.9634E-05 & $\setminus$ & 1.1170E-06 & $\setminus$ & 3.0922E-06 & $\setminus$ \\
  1/16 & 6.9806E-07 & 4.9937 & 9.4196E-06 & 3.9840 & 2.6538E-07 & 4.9957 &  9.3570E-07 & 4.1551 \\
  1/20 & 2.2896E-07 & 4.9956 & 3.8683E-06 & 3.9884 & 8.7028E-08 & 4.9965 &  3.7495E-07 & 4.0983 \\
  1/24 & 9.2071E-08 & 4.9967 & 1.8686E-06 & 3.9909 & 3.4993E-08 & 4.9971 &  1.7861E-07 & 4.0673 \\
\toprule
Scheme       & \multicolumn{2}{c|}{ FVE-4-4 } & \multicolumn{2}{c|}{ FVE-4-3 }
       & \multicolumn{2}{c|}{ FVE4-6 } & \multicolumn{2}{c}{ FE-4 } \\
  \hline
  $h$  & $ \|u-u_{h}\|_{H^{1},x}^{\mathrm{U}} $ & Order & $ \|u-u_{h}\|_{H^{1},x}^{\mathrm{U}} $ & Order
       & $ \|u-u_{h}\|_{H^{1},x}^{\mathrm{U}} $ & Order & $ \|u-u_{h}\|_{H^{1},x}^{\mathrm{U}} $ & Order\\
  \hline
  1/8  & 3.0479e-07 & $\setminus$ & 3.7555e-06 & $\setminus$ & 1.0251e-07 & $\setminus$ & 3.3702e-07 & $\setminus$\\
  1/12 & 2.6831e-08 & 5.9933 & 4.8851e-07 & 5.0303 & 9.1740e-09 & 5.9527 & 4.1182e-08 & 5.1845\\
  1/16 & 4.7817e-09 & 5.9954 & 1.1519e-07 & 5.0220 & 1.6476e-09 & 5.9685 & 9.4836e-09 & 5.1044\\
  1/20 & 1.2659e-09 & 5.9557 & 3.7590e-08 & 5.0186 & 4.4636e-10 & 5.8526 & 3.0657e-09 & 5.0608\\
   \bottomrule
  \end{tabular}
  }
\end{table}

\begin{example}[Ultraconvergence of derivatives]
\label{exam:UltraH1_exam}

Table~\ref{tab:Ultra_k} ($k=3,4$) reports the discrete derivative errors $\|e_h\|_{H^{1},x}^{\mathrm{U}}$ for the FVE-$k$-$r$ schemes and the standard bi-$k$ finite element scheme (FE-$k$). The data reveal the following.

1. When the diffusion $\mathbb{D}$ is diagonal and convection vanishes, FVE schemes satisfying the tensorial $k$-$k$-order orthogonality condition attain $(k+2)$-order derivative ultraconvergence; this is
confirmed for FVE-3-3 and FVE-4-4 on BVP-DR, and for FVE-3-4 and FVE-4-6 on BVP-D.

2. Schemes that violates the tensorial $k$-$k$-order orthogonality condition (FVE-3-2 and FVE-4-3) fail to achieve $(k+2)$-order derivative ultraconvergence, even for the simplest BVP-D case.

3. Enhancing higher-order orthogonality ($r > k$) does not improve derivative ultraconvergence beyond $(k+2)$-order (FVE-3-4 and FVE-4-6 on BVP-D), in contrast to the one-dimensional case where higher-order orthogonality yields improved convergence.

4. Standard bi-$k$-order finite element schemes (FE-3 and FE-4) exhibit no natural derivative ultraconvergence, even for the simplest BVP-D case.
\end{example}

\begin{table}[htbp!]
\centering
  \caption{The superconvergence results of Example~\ref{exam:SuperH1_exam}}\label{tab:SuperH1_exam}
  \resizebox{\textwidth}{!}
  {
  \begin{tabular}{ c | c c | c c |c c | c c }
  \hline
  \toprule
       & \multicolumn{2}{c|}{ FVE-3-2 } & \multicolumn{2}{c|}{ FVE-3-3 }
       & \multicolumn{2}{c|}{ FVE-4-3 } & \multicolumn{2}{c}{ FVE-4-4 } \\
  \hline
  $h$  & $ \|u-u_{h}\|_{ H^{1},x}^{\mathrm{S} }$ & Order & $ \|u-u_{h}\|_{ L^{2}}^{\mathrm{S} } $ & Order
       & $ \|u-u_{h}\|_{ H^{1},x}^{\mathrm{S} }$ & Order & $ \|u-u_{h}\|_{ L^{2}}^{\mathrm{S} } $ & Order\\
  \hline
  1/12 & 1.6249E-04 & $\setminus$ & 1.1808E-06 & $\setminus$ & 4.2542E-06 & $\setminus$ & 1.2039E-08 & $\setminus$\\
  1/16 & 5.1237E-05 & 4.0119 & 2.8391E-07 & 4.9545 & 1.0062E-06 & 5.0116 & 2.1846E-09 & 5.9328\\
  1/20 & 2.0937E-05 & 4.0106 & 9.3771E-08 & 4.9646 & 3.2908E-07 & 5.0085 & 5.7937E-10 & 5.9479\\
  1/24 & 1.0080E-05 & 4.0093 & 3.7885E-08 & 4.9708 & 1.3210E-07 & 5.0063 & 1.9586E-10 & 5.9485\\
   \bottomrule
  \end{tabular}
  }
\end{table}
\begin{example}[Superconvergence of derivatives and function values]\label{exam:SuperH1_exam}
Table~\ref{tab:SuperH1_exam} demonstrates $(k+1)$-order derivative superconvergence for FVE-$k$-$(k-1)$ schemes and $(k+2)$-order function-value superconvergence for FVE-$k$-$k$ schemes on BVP-DQR, consistent with our theoretical predictions.
\end{example}

\section{Conclusion}
\label{sec:Conclusion}
In this paper, we introduce a novel family of $(k+2)$-order derivative ultraconvergence structures for the finite volume element method on rectangular meshes. A key feature is the freedom to place--possibly asymmetric--ultraconvergence points at will, whereas standard bi-$k$ finite element schemes possess no natural $(k+2)$-order derivative ultraconvergence.
Theoretically, we design the asymmetric-enabled M-decomposition for ultraconvergence (AMD-Ultra) to build a superclose function that links the exact and numerical solutions at the chosen points.
Companion asymmetric-enabled structures for one-order-higher derivative and function-value superconvergence are also developed, with the corresponding AMD-Super framework established in parallel.
Numerical experiments validate our theoretical findings and demonstrate the practical utility of the proposed frameworks.


\begin{thebibliography}{10}

\bibitem{Babuska.2007}
I.~Babu{\v{s}}ka, U.~Banerjee, and J.~E. Osborn.
\newblock Superconvergence in the generalized finite element method.
\newblock {\em Numer. Math.}, 107(3):353--395, 2007.

\bibitem{Bramble.1977}
J.~Bramble and A.~Schatz.
\newblock Higher order local accuracy by averaging in the finite element
  method.
\newblock {\em Math. Comp.}, 31(137):94--111, 1977.

\bibitem{Cai.1991}
Z.~Cai, J.~Mandel, and S.~McCormick.
\newblock The finite volume element method for diffusion equations on general
  triangulations.
\newblock {\em SIAM J. Numer. Anal.}, 28(2):392--402, 1991.

\bibitem{Cao.2013}
W.~Cao, Z.~Zhang, and Q.~Zou.
\newblock Superconvergence of any order finite volume schemes for 1{D} general
  elliptic equations.
\newblock {\em J. Sci. Comput.}, 56(3):566--590, 2013.

\bibitem{Cao.2015}
W.~Cao, Z.~Zhang, and Q.~Zou.
\newblock Is 2$k$-conjecture valid for finite volume methods?
\newblock {\em SIAM J. Numer. Anal.}, 53(2):942--962, 2015.

\bibitem{ChenC.2013}
C.~Chen and S.~Hu.
\newblock The highest order superconvergence for bi-{$k$} degree rectangular
  elements at nodes: a proof of {$2k$}-conjecture.
\newblock {\em Math. Comp.}, 82(283):1337--1355, 2013.

\bibitem{Chen.1995}
C.~Chen and Y.~Huang.
\newblock {\em High Accuracy Theory of Finite Element Methods (In Chinese)}.
\newblock {Hunan Science and Technology Publishing House}, 1995.

\bibitem{Chen.2010}
L.~Chen.
\newblock A new class of high order finite volume methods for second order
  elliptic equations.
\newblock {\em SIAM J. Numer. Anal.}, 47(6):4021--4043, 2010.

\bibitem{Chen.2015}
Z.~Chen, Y.~Xu, and Y.~Zhang.
\newblock A construction of higher-order finite volume methods.
\newblock {\em Math. Comp.}, 84(292):599--628, 2015.

\bibitem{Chou.2007}
S.-H. Chou and X.~Ye.
\newblock Superconvergence of finite volume methods for the second order
  elliptic problem.
\newblock {\em Comput. Methods Appl. Mech. Engrg.}, 196(37-40):3706--3712,
  2007.

\bibitem{Cockburn.2025}
B.~Cockburn and Z.~Lal.
\newblock Turbo post-processing for discontinuous {G}alerkin methods:
  one-dimensional linear transport.
\newblock {\em J. Sci. Comput.}, 103(2):58, 2025.

\bibitem{Douglas.1974}
J.~Douglas and T.~Dupont.
\newblock Galerkin approximations for the two point boundary problem using
  continuous, piecewise polynomial spaces.
\newblock {\em Numer. Math.}, 22:99--109, 1974.

\bibitem{Ewing.2002}
R.~E. Ewing, T.~Lin, and Y.~Lin.
\newblock On the accuracy of the finite volume element method based on
  piecewise linear polynomials.
\newblock {\em SIAM J. Numer. Anal.}, 39(6):1865--1888, 2002.

\bibitem{He.2016}
W.~He, R.~Lin, and Z.~Zhang.
\newblock Ultraconvergence of finite element method by {R}ichardson
  extrapolation for elliptic problems with constant coefficients.
\newblock {\em SIAM J. Numer. Anal.}, 54(4):2302--2322, 2016.

\bibitem{Hong.2018}
Q.~Hong and J.~Wu.
\newblock Coercivity results of a modified ${Q}_1$-finite volume element scheme
  for anisotropic diffusion problems.
\newblock {\em Adv. Comput. Math.}, 44(3):897--922, 2018.

\bibitem{Hu.2021}
J.~Hu, L.~Ma, and R.~Ma.
\newblock Optimal superconvergence analysis for the {C}rouzeix-{R}aviart and
  the {M}orley elements.
\newblock {\em Adv. Comput. Math.}, 47(4):1--25, 2021.

\bibitem{Huang.1998}
J.~Huang and S.~Xi.
\newblock On the finite volume element method for general self-adjoint elliptic
  problems.
\newblock {\em SIAM J. Numer. Anal.}, 35(5):1762--1774, 1998.

\bibitem{Krizek.1987}
M.~K{\v{r}}{\'i}{\v{z}}ek and P.~Neittaanm{\"a}ki.
\newblock On superconvergence techniques.
\newblock {\em Acta Appl. Math.}, 9(3):175--198, 1987.

\bibitem{Li.2000}
R.~Li, Z.~Chen, and W.~Wu.
\newblock {\em Generalized Difference Methods for Differential Equations}.
\newblock {Marcel Dekker}, 2000.

\bibitem{Li.2021}
Y.~Li, T.~Zhao, Z.~Zhang, and T.~Wang.
\newblock The high order augmented finite volume methods based on series
  expansion for nonlinear degenerate parabolic equations.
\newblock {\em J. Sci. Comput.}, 88(1), 2021.

\bibitem{Lin.1996}
Q.~Lin and N.~Yan.
\newblock {\em The Construction and Analysis for Efficient Finite Elements (In
  Chinese)}.
\newblock {Hebei University Press House}, 1996.

\bibitem{Lin.2008}
R.~Lin and Z.~Zhang.
\newblock Natural superconvergence points in three-dimensional finite elements.
\newblock {\em SIAM J. Numer. Anal.}, 46(3):1281--1297, 2008.

\bibitem{Lin.2013}
T.~Lin and X.~Ye.
\newblock A posteriori error estimates for finite volume method based on
  bilinear trial functions for the elliptic equation.
\newblock {\em J. Comput. Appl. Math.}, 254:185--191, 2013.

\bibitem{Lin.2015}
Y.~Lin, M.~Yang, and Q.~Zou.
\newblock ${L}^2$ error estimates for a class of any order finite volume
  schemes over quadrilateral meshes.
\newblock {\em SIAM J. Numer. Anal.}, 53(4):2030--2050, 2015.

\bibitem{Lv.2012b}
J.~Lv and Y.~Li.
\newblock ${L}^2$ error estimates and superconvergence of the finite volume
  element methods on quadrilateral meshes.
\newblock {\em Adv. Comput. Math.}, 37(3):393--416, 2012.

\bibitem{Lv.2012}
J.~Lv and Y.~Li.
\newblock Optimal biquadratic finite volume element methods on quadrilateral
  meshes.
\newblock {\em SIAM J. Numer. Anal.}, 50(5):2379--2399, 2012.

\bibitem{Meng.2023}
X.~Meng and M.~Stynes.
\newblock Energy-norm and balanced-norm supercloseness error analysis of a
  finite volume method on {S}hishkin meshes for singularly perturbed
  reaction-diffusion problems.
\newblock {\em Calcolo}, 60(3):40, 2023.

\bibitem{Schatz.1996}
A.~Schatz, I.~Sloan, and L.~Wahlbin.
\newblock Superconvergence in finite element methods and meshes that are
  locally symmetric with respect to a point.
\newblock {\em SIAM J. Numer. Anal.}, 33(2):505--521, 1996.

\bibitem{Sheng.2022}
Z.~Sheng and G.~Yuan.
\newblock Analysis of the nonlinear scheme preserving the maximum principle for
  the anisotropic diffusion equation on distorted meshes.
\newblock {\em Sci. China Math.}, 65(11):2379--2396, 2022.

\bibitem{Suli.1991}
E.~S\"uli.
\newblock Convergence of finite volume schemes for {P}oisson's equation on
  nonuniform meshes.
\newblock {\em SIAM J. Numer. Anal.}, 28(5):1419--1430, 1991.

\bibitem{Thomee.1977}
V.~Thom{\'e}e.
\newblock High order local approximations to derivatives in the finite element
  method.
\newblock {\em Math. Comp.}, 31(139):652--660, 1977.

\bibitem{Wahlbin.1995}
L.~B. Wahlbin.
\newblock {\em Superconvergence in Galerkin Finite Element Methods}.
\newblock Springer, 1995.

\bibitem{Wang.2017}
J.~Wang, C.~Chen, and Z.~Xie.
\newblock The highest superconvergence analysis of {ADG} method for two point
  boundary values problem.
\newblock {\em J. Sci. Comput.}, 70(1):175--191, 2017.

\bibitem{Wang.2016}
X.~Wang and Y.~Li.
\newblock ${L}^2$ error estimates for high order finite volume methods on
  triangular meshes.
\newblock {\em SIAM J. Numer. Anal.}, 54(5):2729--2749, 2016.

\bibitem{Wang.2021b}
X.~Wang, J.~Lv, and Y.~Li.
\newblock New superconvergent structures developed from the finite volume
  element method in 1{D}.
\newblock {\em Math. Comp.}, 90(329):1179--1205, 2021.

\bibitem{Wang.2024}
X.~Wang, Y.~Zhang, and Z.~Zhang.
\newblock Flexible ultra-convergence structures for the finite volume element
  method.
\newblock {\em J. Sci. Comput.}, 101(1):1--27, 2024.

\bibitem{Wang.2025}
X.~Wang, Y.~Zhang, and Z.~Zhang.
\newblock Stability and {$H^1$} estimate for general finite volume element
  method on quadrilateral meshes.
\newblock {\em Preprint}, pages 1--29, 2025.

\bibitem{Xu.2009}
J.~Xu and Q.~Zou.
\newblock Analysis of linear and quadratic simplicial finite volume methods for
  elliptic equations.
\newblock {\em Numer. Math.}, 111(3):469--492, 2009.

\bibitem{Yang.2009}
M.~Yang, C.~Bi, and J.~Liu.
\newblock Postprocessing of a finite volume element method for semilinear
  parabolic problems.
\newblock {\em ESAIM Math. Model. Numer. Anal.}, 43(5):957--971, 2009.

\bibitem{Zhang.2023}
Y.~Zhang and X.~Wang.
\newblock Unified construction and ${L}^2$ analysis for the finite volume
  element method over tensorial meshes.
\newblock {\em Adv. Comput. Math.}, 49(1):2, 2023.

\bibitem{Zhang.1996}
Z.~Zhang.
\newblock Ultraconvergence of the patch recovery technique.
\newblock {\em Math. Comp.}, 65(216):1431--1437, 1996.

\bibitem{Zhang.2012}
Z.~Zhang.
\newblock Superconvergence points of polynomial spectral interpolation.
\newblock {\em SIAM J. Numer. Anal.}, 50(6):2966--2985, 2012.

\bibitem{Zhang.2005}
Z.~Zhang and A.~Naga.
\newblock A new finite element gradient recovery method: superconvergence
  property.
\newblock {\em SIAM J. Sci. Comput.}, 26(4):1192--1213, 2005.

\bibitem{Zhang.2015}
Z.~Zhang and Q.~Zou.
\newblock Vertex-centered finite volume schemes of any order over quadrilateral
  meshes for elliptic boundary value problems.
\newblock {\em Numer. Math.}, 130(2):363--393, 2015.

\bibitem{Zhu.1989}
Q.~Zhu and Q.~Lin.
\newblock {\em The Superconvergence Theory of Finite Elements (In Chinese)}.
\newblock {Hunan Science and Technology Publishing House}, 1989.

\bibitem{Zienkiewicz.1992}
O.~Zienkiewicz and J.~Zhu.
\newblock The superconvergent patch recovery anda posteriori error estimates.
  part 1: The recovery technique.
\newblock {\em Int. J. Numer. Methods Engrg.}, 33(7):1331--1364, 1992.

\end{thebibliography}
\end{document}